\documentclass[12pt]{amsart}

\usepackage{amssymb,amsmath}
\usepackage{amsthm}
\usepackage{tikz}
\usepackage{fancyhdr}
\usepackage[matrix,arrow,tips]{xy}
\usepackage{color}
\usepackage{fullpage}
\usepackage{booktabs}
\usepackage{enumerate}

\usetikzlibrary{patterns}
\usetikzlibrary{calc}

\newcommand{\eps}{\epsilon}

\newcommand{\bbC}{\mathbb{C}}
\newcommand{\bbD}{\mathbb{D}}

\newcommand{\bbZ}{\mathbb{Z}}
\newcommand{\bbM}{\mathbb{M}}

\newcommand{\bbR}{\mathbb{R}}

\newcommand{\calD}{\mathcal{D}}
\newcommand{\calE}{\mathcal{E}}

\newcommand{\calM}{\mathcal{M}}

\newcommand{\calP}{\mathcal{P}}
\newcommand{\calR}{\mathcal{R}}

\newcommand{\calW}{\mathcal{W}}

\newcommand{\bfH}{\mathbf{H}}

\newcommand{\del}{\partial}

\newcommand{\tr}{\mathrm{tr}}
\newcommand{\End}{\mathrm{End}}

\newcommand{\diag}{\mathrm{diag}}

\newcommand{\laa}{\mathfrak{a}}

\newcommand{\lah}{\mathfrak{h}}

\newcommand{\lau}{\mathfrak{u}}

\newcommand{\GL}{\mathrm{GL}}

\newcommand{\SO}{\mathrm{SO}}
\newcommand{\Sp}{\mathrm{Sp}}

\newcommand{\SU}{\mathrm{SU}}

\newcommand{\rank}{\mathrm{rank}}

\renewenvironment{proof}{\noindent{\scshape Proof.}}{\qed}

\theoremstyle{plain}
\newtheorem{theorem}{Theorem}[section]
\newtheorem{lemma}[theorem]{Lemma}
\newtheorem{proposition}[theorem]{Proposition}
\newtheorem{corollary}[theorem]{Corollary}

\newtheorem{definition}[theorem]{Definition}

\theoremstyle{definition}
\newtheorem{example}[theorem]{Example}
\newtheorem{remark}[theorem]{Remark}

\title{Vector-Valued Heckman-Opdam Polynomials:\\a Steinberg variation}

\author{Maarten van Pruijssen}

\begin{document}
\begin{abstract}
We develop a theory of Jacobi polynomials for parabolic subgroups of finite reflection groups that specializes to the cases studied by Heckman and Opdam in which the whole group and the trivial group are considered. For the intermediate cases we combine results of Steinberg and Heckman and Opdam to obtain new examples of families of vector-valued orthogonal polynomials with properties similar to those of the usual Jacobi polynomials. Most notably we show that these polynomials, when suitably interpreted as vector-valued polynomials, are determined up to scaling as simultaneous eigenfunctions of a commutative algebra of differential operators. 
We establish an example in which the vector-valued Jacobi polynomials can be identified with spherical functions for a higher $K$-type on a compact symmetric pair with restricted root system of Dynkin type $A_{2}$. We also describe how to obtain new examples of matrix-valued orthogonal polynomials in several variables.
\end{abstract}
\maketitle

\tableofcontents

\newpage

\section{Introduction}\label{s:intro}

Let $\laa$ be a Euclidean space with inner product $(\cdot,\cdot)$, $R\subset\laa^{*}$ a root system with Weyl group $W$ and $k=(k_{\alpha})_{\alpha\in R}$ a multiplicity function that is $W$-invariant. Let $P\subset\laa^{*}$ denote the weight lattice and $Q^{\vee}\subset\laa^{*}$ the coroot lattice generated by the coroots $\alpha^{\vee}=2\alpha/(\alpha,\alpha)$. The group algebra $\bbC[P]$ is identified with the space of Laurent polynomials on the compact torus $A=i\laa/(2\pi iQ^{\vee})$ and carries a natural action of $W$, as follows. For $w\in W, \lambda\in P$ and $e^{\lambda}\in\bbC[P]$ we have $we^{\lambda}=e^{w\lambda}$.
Fix a a subset $R^{+}\subset R$ of positive roots. Let $\Pi\subset R^{+}$ be the set of simple roots and $S\subset W$ the corresponding set of simple reflections. Let $I\subset S$ be a subset and denote by $W_{I}$ the parabolic subgroup generated by the simple reflections in $I$.
In this paper we study the harmonic analysis of $\bbC[P]^{W_{I}}$, the space of $W_{I}$-invariant Laurent polynomials on $A$.

Let $R_{I}\subset R$ denote the set of roots corresponding to $W_{I}$ and let $R_{I}^{+}=R_{I}\cap R^{+}$ be the set of positive roots of $W_{I}$ that is compatible with the choice of $R_{+}$. Let $P^{+}\subset P$ denote the set of dominant weights with respect to $R^{+}$. 
The action of $W_{I}$ on $P$ has a unique fundamental domain $P_{I}^{+}\subset P$ that contains $P^{+}$. 
Let $W^{I}=\{v\in W\mid v(R_{I}^{+})\subset R_{+}\}$, which is a system of representatives of $W/W_{I}$ consisting of the shortest elements of each coset. Using the Steinberg weights below, we show that there is an isomorphism $W^{I}\times P_{+}\cong P^{+}_{I}$.
The space $\bbC[P]^{W_{I}}$ is a freely and finitely generated $\bbC[P]^{W}$-module and the generators $\phi_{v},v\in W^{I}$ can be made explicit. 

The space $\bbC[P]$ is equipped with the usual inner product $(\cdot,\cdot)_{k}$ (cf.~\cite[\S2]{MR1353018}) that we restrict to $\bbC[P]^{W_{I}}$. The latter space has a natural vector space basis labeled by $P_{I}^{+}$. At the same time $P_{I}^{+}$ inherits the partial ordering on $P$ introduced by Heckman (cf.~Def 2.4 of loc.cit.) that allows us to define the Jacobi polynomials $p_{I}(\lambda,k), \lambda\in P^{+}_{I}$ which constitute a vector space basis of $\bbC[P]^{W_{I}}$. In the extreme cases $W_{\emptyset}=\{e\}$ and $W_{S}=W$, our Jacobi polynomials coincide with the non-symmetric polynomials $E(\lambda,k)$ from Def.~2.6 of loc.cit.~
(introduced by Heckman) and the $W$-invariant Jacobi polynomials $P(\lambda,k),\lambda\in P_{+}$ from \cite[Def.~1.3.1]{MR1313912} respectively. 
We show that the Jacobi polynomials can also be obtained by $W_{I}$-symmetrizing the non-symmetric Jacobi polynomials $E(\lambda,k)$ with $\lambda\in P_{I}^{+}$. 

Denote by $\bfH(R_{+},k)$ Lusztig's graded Hecke algebra that acts on $\bbC[P]$ and denote by $S'(\lah)\subset\bfH(R_{+},k)$ the commutative subalgebra generated by the Cherednik operators $D_{\xi}$, where $\lah=\laa\otimes\bbC$. The algebra $S'(\lah)^{W_{I}}$ of $W_{I}$-invariant operators acts on the space $\bbC[P]^{W_{I}}$ by differential-reflection operators. 
The space $\bbC[P]^{W_{I}}$ can be identified with the space of invariant vector-valued Laurent polynomials, $\bbC[P]^{W_{I}}\cong\left(\bbC[P]\otimes\bbC[W/W_{I}]\right)^{W}$. The corresponding action of the elements in $S'(\lah)$ is now by differential operators, the coefficients being matrix-valued Laurent polynomials. Note that these considerations break down if we consider, more generally, a reflection subgroup, because the above observations rely on relations in $\bfH(R_{+},k)$ for simple reflections.
The idea to make $W$-invariant vector-valued polynomials out of $W_{I}$-invariant polynomials was already used by Steinberg \cite[L.2.4]{MR372897}. At the same time, $W$-invariant vector-valued Laurent polynomials also occur naturally when restricted spherical functions are considered, see the first application below. 

We point out two applications of this class of Jacobi polynomials . For the first application we recall that the Jacobi polynomials $p_{S}(\lambda,k)$ can be identified with the zonal spherical functions of a compact symmetric pair whose restricted root system is of type $R$, once the multiplicity $k$ is related to the dimensions of the restricted root spaces. The root system $R$ of type $A_{2}$ occurs as the restricted root system for the compact symmetric pairs
$$(\SU(3),\SO(3)),\quad(\SU(3)\times\SU(3),\diag(\SU(3)),\quad(\SU(6),\Sp(6)).$$

The spherical functions related to the defining representation of the symmetric subgroup in each case, can be identified with the Jacobi polynomials $p_{\{s\}}(\lambda,k)$, where $s\in W$ is a simple reflection. Our identification is based on the comparison of generators of $S'(\lah)^{<s>}$, viewed as differential operators acting on the vector-valued Laurent polynomials, with the differential operators found by Shimeno \cite{MR3775397}.

The second application is a link to the theory of vector- and matrix-valued orthogonal polynomials in several variables. Only in this discussion we make the distinction between Laurent polynomials in $\bbC[P]$ and genuine polynomials in $\bbC[P]^{W}$ (which is a polynomial ring) whereas usually we refer to elements in $\bbC[P]$ as polynomials.  

The isomorphism $\bbC[P]^{W_{I}}\cong \bbC[P]^{W}\otimes\bbC^{|W^{I}|}$ based on Steinberg's result gives an interpretation of the Jacobi polynomials as a family of vector-valued orthogonal polynomials. The labeling of the orthogonal polynomials suggests that we can pack them together to obtain matrix-valued orthogonal polynomials. In this way we obtain new examples families of orthogonal vector- and matrix-valued polynomials that are at the same uniquely determined (up to scaling) as simultaneous eigenfunctions of a commutative algebra of differential operators.
In rank one this interpretation is a reasonable cosmetic operation leading to Gauss' hypergeometric functions, but for higher rank it seems an unsatisfying descriptions of the Jacobi polynomials.

Historically, in the research of matrix-valued orthogonal polynomials, it has been fruitful to try to find analogs of properties of scalar-valued orthogonal polynomials. For example, for a while it has been an open question whether families of matrix-valued orthogonal polynomials exist  for which there is a second-order differential operator that is symmetric and that has the polynomials as simultaneous eigenfunctions.

The first examples that answer this question affirmatively came about using spherical functions on compact Lie groups, see \cite{MR4053617}. These examples have have been vastly generalized in \cite{PvP} using spherical varieties and their combinatorics.
In this framework there are examples of families of matrix-valued orthogonal polynomials in any number of variables and of sizes that can become arbitrarily large. However, the parameters of the polynomials from this theory are discrete, being dimensions of root spaces. Also, the commutative algebra of differential operators remains rather abstract as a subquotient of the universal enveloping algebra. 

The new class of vector- and matrix-valued orthogonal polynomials from the theory in this paper has the advantage that the multiplicity function is more flexible and that the algebra of differential operators is better understood. Many interesting questions about the properties of these new Jacobi polynomials, for example about their various normalizations and closer ties with the representation theory of compact Lie groups, remain to be investigated. 

The paper is organized as follows. In Section \ref{s:LaurentPoly} we establish a basis of standard $W_{I}$-invariant polynomials, partly based on Steinberg's observations. In Section \ref{s:JacobiPoly} we consider a partial ordering on the labels of the invariant polynomials and we apply the Gram-Schmidt process to make up a new basis, for which it is \textit{a priori} not clear that polynomials with incomparable labels are orthogonal. This issue is resolved in Section \ref{s:ortho} using differential-reflection operators. On the way we have to establish a couple of results where the partial ordering is understood relative to the parabolic subgroup $W_{I}$. In Section \ref{s:vvJacobi} we make vector-valued Laurent polynomials out of the invariant polynomials and observe that the reflection part of the differential-reflection operators translates into matrix-coefficients of the differential operators. In this section we also calculate a specific example that is not one of the extreme cases. We pick up that example in Section \ref{s:AppI} to show that the vector-valued Laurent-polynomials can be identified with the spherical functions on a specific compact Riemann symmetric pair $(U,K)$ with specific irreducible $K$-representation. In Section \ref{s:AppII} we formulate a recipe to make vector- and matrix-valued orthogonal polynomials that are determined as simultaneous eigenfunctions (up to scaling) of a commutative algebra of differential operators.

\section{Invariant Laurent polynomials}  \label{s:LaurentPoly}

We retain the notation of Section \ref{s:intro}. The space of invariant polynomials $\bbC[P]^{W_{I}}$ is naturally a module over the ring $\bbC[P]^{W}$ of $W$-invariants in $\bbC[P]$. 
Due to Steinberg \cite[Thm.2.2]{MR372897} this module structure is completely understood.
For $\alpha\in\Pi$ let $\varpi_{\alpha}\in P$ be the corresponding fundamental weight, i.e.~$(\varpi_{\alpha},\beta^{\vee})=\delta_{\alpha,\beta}$ for $\alpha,\beta\in\Pi$. For $v\in W^{I}$ we call
$$\lambda_{v}=\sum_{\alpha\in\Pi, v^{-1}\alpha<0}\varpi_{\alpha}\in P^{+}$$
the \textit{Steinberg weight} associated to $v\in W^{I}$.
Let $W_{I,v^{-1}\lambda_{v}}=\{w\in W_{I}\mid wv^{-1}\lambda_{v}=v^{-1}\lambda_{v}\}$. Given $\alpha\in R_{I}^{+}$ we have $(v^{-1}\lambda_{v},\alpha^{\vee})\ge0$ which implies that $W_{I,v^{-1}\lambda_{v}}\subset W_{I}$ is a parabolic subgroup. For each Steinberg weight we define the polynomial
\begin{eqnarray*}
\phi_{v}=\sum_{w\in W_{I}/W_{I,v^{-1}\lambda_{v}}}we^{v^{-1}\lambda_{v}}\in\bbC[P]^{W_{I}}.\label{Steinberg polynomial}
\end{eqnarray*}
\begin{theorem}[Steinberg]
The space $\bbC[P]^{W_{I}}$ of invariant Laurent polynomials is a free $\bbC[P]^{W}$-module generated by the elements $\phi_{v}$ with $v\in W^{I}$. 
\end{theorem}

Let $P_{I}^{+}=\{\lambda\in P\mid (\lambda,\alpha)\ge0\mbox{ for all $\alpha\in R_{I}^{+}$}\}$. Then $P_{I}^{+}$ is a strict fundamental domain for the action of $W_{I}$ on $P$, i.e.~$W_{I}\cdot P_{I}^{+}=P$ and for $\lambda\in P_{I}^{+}$ we have $W_{I}\cdot\lambda\cap P_{I}^{+}=\{\lambda\}$. Moreover, $P_{I}^{+}$ is the unique strict fundamental domain that contains $P^{+}$. For $\lambda\in P$ we denote by $\lambda_{+}$ the unique element in $W\cdot\lambda\cap P^{+}$.

If $\alpha\in R_{I}^{+}$ and $(v,\sigma)\in W^{I}\times P^{+}$, then $(v^{-1}(\lambda_{v}+\sigma),\alpha)=(\lambda_{v}+\sigma,v\alpha)\ge0$ because $v\alpha\in R^{+}$. This defines a map
$f_{I}:W^{I}\times P^{+}\to P_{I}^{+}, f_{I}(v,\sigma)=v^{-1}(\lambda_{v}+\sigma)$.

\begin{proposition}\label{PplusSteinberg}
The map $f_{I}$ is a bijection.
\end{proposition}

\begin{proof}
We first show that $f_{\emptyset}$ is a bijection.
The sets $v^{-1}(\lambda_{v}+P^{+})$ are pairwise disjoint for $v\in W$. Indeed, suppose that $v^{-1}(\lambda_{v}+\sigma)=u^{-1}(\lambda_{u}+\tau)$. Then $\lambda_{v}+\sigma=\lambda_{u}+\tau$ and $uv^{-1}$ stabilizes this element. Then $uv^{-1}\in W_{\lambda_{v}+\sigma}=W_{\lambda_{u}+\tau}\subset W_{\lambda_{u}}$. So $uv^{-1}$ is in the group generated by reflections in simple roots that are perpendicular to $\lambda_{u}$. These are the simple roots that are made negative by $u^{-1}$. Hence $u^{-1}\in W^{\lambda_{u}}$ and we have $\ell(u^{-1})+\ell(uv^{-1})=\ell(v^{-1})$. The roles of $u$ and $v$ may be reversed to obtain also $\ell(v^{-1})+\ell(vu^{-1})=\ell(u^{-1})$ from which $\ell(uv^{-1})=0$ whence $u=v$ and $\sigma=\tau$.

For $\lambda\in P$ we denote by $\overline{v}(\lambda)$ the shortest element in $W$ that sends $\lambda_{+}$ to $\lambda$. Note that $\overline{v}(\lambda)\in W^{\lambda_{+}}$.
Let $v=\overline{v}(\lambda)^{-1}$.
We claim that there exists $\sigma\in P^{+}$ such that $\lambda=v^{-1}(\lambda_{v}+\sigma)$. This is equivalent to $v\lambda-\lambda_{v}\in P^{+}$ and in turn to $(v\lambda,\alpha^{\vee})>0$ for all $\alpha\in\Pi$ for which $v^{-1}\alpha<0$. Suppose that $(v\lambda,\alpha^{\vee})=0$ for some $\alpha\in\Pi$ with $v^{-1}\alpha<0$. This implies $s_{\alpha}\in W_{\lambda_{+}}$ and $\ell(\overline{v}(\lambda)s_{\alpha})<\ell(\overline{v}(\lambda))$ and in turn that $\overline{v}(\lambda)s_{\alpha}$ is a coset representative of $\overline{v}(\lambda)W_{\lambda_{+}}$ whose length is smaller than that of $\overline{v}(\lambda)$, a contradiction to $\overline{v}(\lambda)\in W^{\lambda_{+}}$.

This shows that $f_{\emptyset}:W\times P^{+}\to P$ is a bijection. Since $f_{I}=f_{\emptyset}|_{W^{I}\times P^{+}}$, the map $f_{I}$ is injective.

If $\lambda\in P_{I}^{+}$, then $\lambda=\overline{v}(\lambda)(\lambda_{\overline{v}(\lambda)^{-1}}+\sigma)$ for $\overline{v}(\lambda)\in W^{\lambda_{+}}$.
We have to show that $\overline{v}(\lambda)^{-1}\in W^{I}$. Note that $\overline{v}(\lambda)^{-1}\alpha<0$ precisely if $(\lambda,\alpha)<0$ (by \cite[(2.7.2)(ii)]{MR1976581}), so for $\alpha\in R^{+}_{I}$ we have $(\lambda,\alpha)\ge0$ and in turn $\overline{v}(\lambda)^{-1}\alpha>0$.
\end{proof}

\begin{example}\label{exampleGenerators}
To illustrate Steinberg's result, we work out some of the details for the root system of type $A_{2}$ with $W$ being generated by the simple reflections $s_{1},s_{2}$. We consider the subsets $I=\emptyset,\{s_{2}\}, S$.
\begin{itemize}
\item If $I=\emptyset$ then $W_{\emptyset}=\{e\}, W^{\emptyset}=W=S_{3}$ and $P$ can be written as a disjoint union of six copies of $P^{+}$. Each copy is of the form $v^{-1}(\lambda_{v}+P^{+})$, $v\in W$.
\item If $I=\{s_{2}\}$ then $W_ {I}=<s_{2}>$ and $W^{I}=\{e,s_{1},s_{2}s_{1}\}$ and the fundamental domain of the action of $<s_{2}>$ on $P$ is the union of disjoint sets $v^{-1}(\lambda_{v}+P^{+})$, $v\in W^{I}$.
\item If $I=S$ then $W_{S}=W$ and $W^{S}=\{e\}$ the fundamental domain of the action of $W$ on $P$ is $P^{+}$.
\end{itemize}
We have depicted the relevant data in Figure \ref{figure: A2example}. If $I=\{s_{2}\}$, then the $\bbC[P]^{W}$-module $\bbC[P]^{W_{I}}$ is generated by the polynomials $\phi_{e}=1,\phi_{s_{1}}=e^{-\varpi_{1}+\varpi_{2}}+e^{-\varpi_{2}}$ and $\phi_{s_{2}s_{1}}=e^{-\varpi_{1}}$.
\end{example}

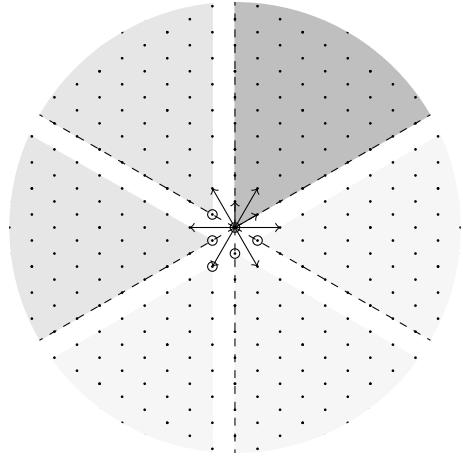
\begin{figure}
\begin{minipage}{.5\textwidth}
\begin{tabular}{ |l|l|l|c|c|c| }
\hline
$v$ & $\lambda_{v}$ & $v^{-1}\lambda_{v}$ & $\emptyset$ &  $\{s_{2}\}$ & $S$\\
\hline \hline
$e$ & $0$ & $0$ & $*$ & $*$ & $*$\\ \hline
$s_{1}$ & $\varpi_{1}$ & $-\varpi_{1}+\varpi_{2}$&$*$& $*$ & \\ \hline
$s_{2}$ & $\varpi_{2}$ & $\varpi_{1}-\varpi_{2}$ &$*$&  & \\\hline
$s_{2}s_{1}$ & $\varpi_{2}$ & $-\varpi_{1}$ &$*$& $*$ & \\\hline
$s_{1}s_{2}$ & $\varpi_{1}$ & $-\varpi_{2}$ &$*$&  & \\\hline
$s_{1}s_{2}s_{1}$ & $\varpi_{1}+\varpi_{2}$ & $-\varpi_{1}-\varpi_{2}$&$*$&  &\\
\hline 

\end{tabular}
\end{minipage}
\begin{minipage}{0.4\textwidth}

\begin{tikzpicture}[scale=.6]
\pgfmathsetmacro\ax{1}
\pgfmathsetmacro\ay{0}
\pgfmathsetmacro\bx{1 * cos(120)}
\pgfmathsetmacro\by{1 * sin(120)}
\pgfmathsetmacro\lax{2*\ax/3 + \bx/3}
\pgfmathsetmacro\lay{2*\ay/3 + \by/3}
\pgfmathsetmacro\lbx{\ax/3 + 2*\bx/3}
\pgfmathsetmacro\lby{\ay/3 + 2*\by/3}


\begin{scope}

\clip (0,0) circle (5);

\draw [fill,lightgray] (0,0) -- (10*\lax,10*\lay) -- (10*\lbx,10*\lby) -- cycle;

\draw [fill,lightgray!40!] (-1*\lax+\lbx,-1*\lay+\lby) -- (-1*\lax+11*\lbx,-1*\lay+11*\lby) -- (-11*\lax+11*\lbx,-11*\lay+11*\lby) --cycle;

\draw [fill,lightgray!40!] (-\lax,-\lay) -- (-11*\lax,-11*\lay) -- (-11*\lax+10*\lbx,-11*\lay+10*\lby) -- cycle;

\draw [fill,lightgray!15!] (-\lax-\lbx,-\lay-\lby) -- (-\lax-11*\lbx,-\lay-11*\lby) -- (-11*\lax-\lbx,-11*\lay-\lby) -- cycle;

\draw [fill,lightgray!15!] (-\lbx,-\lby) -- (-11*\lbx,-11*\lby) -- (-11*\lbx+10*\lax,-11*\lby+10*\lay) -- cycle;

\draw [fill,lightgray!15!] (\lax-\lbx,\lay-\lby) -- (11*\lax-\lbx,11*\lay-\lby) -- (11*\lax-11*\lbx,11*\lay-11*\lby) -- cycle;

\foreach \k in {1,...,6} 
  \draw[dashed] (0,0) -- (\k * 60 + 30:12);

\draw[thin,->] (0,0) -- (\ax,\ay) ;
\draw[thin,->] (0,0) -- (\bx,\by) ;
\draw[thin,->] (0,0) -- (\lax,\lay);
\draw[thin,->] (0,0) -- (\lbx,\lby);
\draw[thin,->] (0,0) -- (-\ax,-\ay);
\draw[thin,->] (0,0) -- (-\bx,-\by);
\draw[thin,->] (0,0) -- (\ax+\bx,\ay+\by);
\draw[thin,->] (0,0) -- (-\ax-\bx,-\ay-\by);

\draw[] (0,0) circle (3pt);
\draw[] (-1*\lax+\lbx,-1*\lay+\lby) circle (3pt);
\draw[] (-\lax,-\lay) circle (3pt);
\draw[] (-\lax-\lbx,-\lay-\lby) circle (3pt);
\draw[] (-\lbx,-\lby) circle (3pt);
\draw[] (\lax-\lbx,\lay-\lby) circle (3pt);

\foreach \v in {-12,-9,...,10}
\foreach \w in {-12,-9,...,10}
   \draw[fill] 
(\w*\lbx+\v*\lax,\w*\lby+\v*\lay) circle (0.5pt);

\foreach \v in {-10,-9,...,10}
\foreach \w in {-10,-9,...,10}
   \draw[fill] 
(\w*\lbx+\v*\lax,\w*\lby+\v*\lay) circle (0.5pt);

\end{scope}

\end{tikzpicture}

\end{minipage}

\caption{For $A_{2}$ we depict the sets $v^{-1}(\lambda_{v}+P^{+})$ for $v\in W^{I}$. The union of these sets constitute a fundamental domain for the action of $W_{I}$ on $P$. The relevant elements are tabulated on the left while the copies of $P^{+}$ are depicted in different gray tones to indicate to which fundamental domain they contribute for the cases $I=\emptyset,\{s_{2}\},S$ (from light to dark). The $*$ in the last three columns indicates which elements are in $W^{I}$.}
\label{figure: A2example}
\end{figure}

\begin{definition}
For $\lambda\in P^{+}_{I}$ we define $m_{I}(\lambda)=\sum_{\mu\in W_{I}\cdot\lambda}e^{\mu}\in\bbC[P]^{W^{I}}$, the basic $W_{I}$-invariant polynomials.
\end{definition}

\begin{proposition}
The basic invariant polynomials $m_{I}(\lambda),\lambda\in P_{I}^{+}$ constitute a vector space basis of $\bbC[P]^{W_{I}}$.
\end{proposition}

\begin{proof}
The basic invariant polynomials are linearly independent. To see that they span $\bbC[P]^{W_{I}}$ we use induction on the number of terms $e^{\lambda}$ with $\lambda\in P^{+}_{I}$ and a non-zero coefficient in a given polynomial $m\in \bbC[P]^{W_{I}}$.
\end{proof}


\section{Jacobi polynomials}  \label{s:JacobiPoly}

Following \cite[\S2]{MR1353018} we endow $\bbC[P]$ with a Hermitean inner product $(\cdot,\cdot)_{k}$ given by 
\begin{equation}\label{eq:innerproduct}
(\phi,\psi)_{k}=\int_{T}\overline{\phi(t)}\psi(t)\delta_{k}(t)dt,\quad\phi,\psi\in\bbC[P]
\end{equation}
where $\delta_{k}=\prod_{\alpha\in R}(e^{\alpha/2}-e^{-\alpha/2})^{k_{\alpha}}$ and $k_{\alpha}\ge0$.

To obtain a basis of orthogonal Laurent polynomials we use the following partial ordering on $P_{I}^{+}$. Recall that the subset $Q^{+}$ of the root lattice $Q$ that consists of elements $\sum_{\alpha\in\Pi}n_{\alpha}\alpha,n_{\alpha}\in\bbZ_{\ge0}$ induces a partial ordering on $P$,
$$\lambda\preceq\mu\quad\mbox{if and only if}\quad \mu-\lambda\in Q^{+}.$$
Let $\lambda_{+}$ be the unique element in $W\cdot\lambda\cap P^{+}$. Recall that $\overline{v}(\lambda)\in W^{\lambda_{+}}$ is the unique element with $\overline{v}(\lambda)\lambda_{+}=\lambda$.
Given $\lambda,\mu\in P$, we say
$$\lambda\le_{\emptyset}\mu\Longleftrightarrow\left\{\begin{array}{cl}
(1)&\lambda_{+}\preceq\mu_{+},\\
(2)&\mbox{if $\lambda_{+}=\mu_{+}$ then $\overline{v}(\lambda)\le_{W}\overline{v}(\mu)$ in the Bruhat ordering $\le_{W}$ on $W$.}
\end{array}\right.$$
The the restriction of the partial ordering $\le_{\emptyset}$ to $P^{+}_{I}$ is denoted by $\le_{I}$. Note that $\le_{S}$ corresponds to $\preceq$ restricted to $P^{+}$. If we write $\mu\le_{I}\lambda$, then we implicitly assume that both $\mu,\lambda\in P^{+}_{I}$. In particular, for $\lambda\in P_{I}^{+}$ the set $\{\mu\in P_{I}^{+}\mid\mu\le_{I}\lambda\}$ is finite.

\begin{definition}
Let $p_{I}(\lambda,k)\in\bbC[P]^{W_{I}}$ be defined by
\begin{itemize}
\item $p_{I}(\lambda,k)=\sum_{\mu\le_{I}\lambda}c_{\lambda,\mu}m_{I}(\mu)$ with $c_{\lambda,\lambda}=1$,
\item $(p_{I}(\lambda,k),m_{I}(\mu))_{k}=0$ for all $\mu<_{I}\lambda$.
\end{itemize}
\end{definition}

\begin{remark}
In the extreme cases $I=\emptyset$ and $I=S$ the polynomials $p_{I}(\lambda,k)$ are known and constitute an orthogonal basis of the space of $W_{I}$-invariant polynomials.   
\begin{itemize}
\item If $I=\emptyset$ then $p_{\emptyset}(\lambda,k)=E(\lambda,k)$ see \cite[Def.~2.6]{MR1353018}.
\item If $I=S$ then $p_{S}(\lambda,k)=P(\lambda,k)$, the usual $W$-invariant Jacobi polynomials from \cite[Def.~1.3.1]{MR1313912}.
\end{itemize}
The orthogonality of the $E(\lambda,k)$ is established in \cite[Cor.~2.11]{MR1353018} and that of the $P(\lambda,k)$ in \cite[Cor.1.3.13]{MR1313912} using differential-reflection operators. The difficult part is to show that polynomials with incomparable labels are also pairwise orthogonal. We shall establish the corresponding result for the $p_{I}(\lambda,k)$ in Section \ref{s:ortho}.
\end{remark}

We proceed to write $p_{I}(\lambda,k)$ as a $W_{I}$-symmetrized sum of $E(\lambda,k)$ for which we first introduce some notation and some auxiliary results.

If $\lambda\in P^{+}_{I}$, then the stabilizer $W_{I,\lambda}$ of $\lambda$ in $W_{I}$ is a parabolic subgroup with positive roots $R_{I,\lambda}^{+}=\{\alpha\in R_{I}^{+}\mid(\lambda,\alpha)=0\}$. Furthermore, if $\lambda\in P$ then there is a unique element $\lambda_{I,+}\in P_{I}^{+}\cap W_{I}\cdot\lambda$. The element $w$ in $W_{I}$ of minimal length for which $w\lambda_{I,+}=\lambda$ is denoted by $\overline{v}_{I}(\lambda)$.

\begin{lemma}\label{lemma:decomp}
For $\lambda\in P$ we have $\overline{v}(\lambda)^{-1}=\overline{v}(\lambda_{I,+})^{-1}\overline{v}_{I}(\lambda)^{-1}$.
\end{lemma}

\begin{proof}
Write $\overline{v}(\lambda)^{-1}=w''w'\in W^{I}W_{I}$ and note that $(w'')^{-1}\lambda_{+}\in P_{I}^{+}$. It follows that $(w'')^{-1}\lambda_{+}=w'\lambda=\lambda_{I,+}\in P_{I}^{+}$. At the same time $\overline{v}(\lambda_{I,+})^{-1}\overline{v}_{I}(\lambda)^{-1}\in W^{I}W_{I}$ is an element that sends $\lambda$ to $\lambda_{+}$, hence $\ell(\overline{v}(\lambda_{I,+})^{-1}\overline{v}_{I}(\lambda)^{-1})=\ell(\overline{v}(\lambda_{I,+})^{-1})+\ell(\overline{v}_{I}(\lambda)^{-1})\ge \ell((w'')^{-1})+\ell((w')^{-1})$. If the inequality is strict, then $\ell((w'')^{-1})<\ell(\overline{v}(\lambda_{I,+})^{-1})$ or $\ell((w')^{-1})<\ell(\overline{v}_{I}(\lambda)^{-1})$. Both inequalities contradict the minimality of the lengths of the elements $\overline{v}(\lambda_{I,+})$ and $\overline{v}_{I}(\lambda)$, so we have $w''=\overline{v}(\lambda_{I,+})^{-1}$ and $w'=\overline{v}_{I}(\lambda)^{-1}$.
\end{proof}

Let $P^{-}=-P^{+}$ and similarly $P^{-}_{I}=-P^{+}_{I}$. For $\lambda\in P$ we denote by $v(\lambda)\in W$ the shortest element with the property $v(\lambda)\lambda=\lambda_{-}$, the unique element in $W\cdot\lambda\cap P^{-}$. For $w\in W$ we denote $R(w)=R^{+}\cap w^{-1}R^{-}$, where $R^{-}=-R^{+}$. We have
$$R(v(\lambda))=\{\alpha\in R^{+}\mid(\lambda,\alpha)>0\}$$
by \cite[(2.4.4)]{MR1976581}, which will be used in the proof of the following result.

\begin{lemma}\label{lemma:lambdaordering}
Let $\lambda\in P_{I}^{+}$. If $\mu<_{\emptyset}\lambda$ and $w\in W_{I,\lambda}$, then $w\mu<_{\emptyset}\lambda$.
\end{lemma}

\begin{proof}
If $\mu_{+}\prec\lambda_{+}$ then $w\mu<_{\emptyset}\lambda$ for $w\in W$, so we assume $\mu_{+}=\lambda_{+}$.
Let $\alpha\in R_{I,\lambda}^{+}$ be a simple root. The claim follows by induction on $\ell(w)$ once we have $s_{\alpha}\mu<_{\emptyset}\lambda$. 

If $(\mu,\alpha^{\vee})=0$, then $s_{\alpha}(\mu)=\mu$ and the claim holds. If $(\mu,\alpha^{\vee})\ne0$ then $v(s_{\alpha}\mu)=v(\mu)s_{\alpha}$ (by \cite[(2.4.14(i))]{MR1976581}).

Suppose that $(\mu,\alpha^{\vee})<0$. Then $v(\mu)\alpha>0$ (because $\alpha\not\in R(v(\mu))$), which implies $\ell(v(\mu)s_{\alpha})>\ell(v(\mu))$ and in turn $v(\mu)<_{W}v(\mu)s_{\alpha}=v(s_{\alpha}\mu)$. It follows that $v(\lambda)<_{W}v(s_{\alpha}\mu)$ and hence $s_{\alpha}\mu<_{\emptyset}\lambda$.

Suppose that $(\mu,\alpha^{\vee})>0$. Then $v(\mu)\alpha<0$ from which $\ell(v(\mu)s_{\alpha})<\ell(v(\mu))$. At the same time we have $(\lambda,\alpha)=0$ so $\alpha\not\in S(v(\lambda))$ from which $v(\lambda)\alpha>0$. It follows that $\ell(v(\lambda)s_{\alpha})>\ell(v(\lambda))$. The Lifting Property \cite[Prop.2.2.7]{MR2133266} implies $v(\lambda)\le_{W}v(\mu)s_{\alpha}=v(s_{\alpha}\mu)$ and hence $s_{\alpha}\mu\le_{\emptyset}\lambda$. If $s_{\alpha}\mu=\lambda$, then also $\mu=\lambda$, which violates $\mu<_{\emptyset}\lambda$. We conclude that $s_{\alpha}\mu<_{\emptyset}\lambda$. 
\end{proof}

\begin{lemma}
If $w\in W_{I,\lambda}$ and $\lambda\in P^{+}_{I}$, then $wE(\lambda,k)=E(\lambda,k)$.
\end{lemma}

\begin{proof}
Write $E(\lambda,k)=e^{\lambda}+\sum_{\mu<_{\emptyset}\lambda}c_{\mu}e^{\mu}$ and use Lemma \ref{lemma:lambdaordering} to see that $wE(\lambda,k)=e^{\lambda}+\sum_{\mu<_{\emptyset}\lambda}d_{\mu}e^{\mu}$ has the same type of expansion as $E(\lambda,k)$.
If $\nu<_{\emptyset}\lambda$, then
$$(wE(\lambda,k),e^{\nu})_{k}=(E(\lambda,k),e^{w^{-1}\nu})_{k}=0,$$
because $w^{-1}\nu<_{\emptyset}\lambda$ (Lemma \ref{lemma:lambdaordering}) and the defining property of $E(\lambda,k)$. We conclude that $wE(\lambda,k)=E(\lambda,k)$.
\end{proof}

\begin{lemma}\label{lemma:relativeordering}
Let $\lambda\in P^{+}_{I}$ and $\mu\in P$ with $\mu_{+}=\lambda_{+}$. If $\mu <_{\emptyset}\lambda$, then $\mu_{I,+}<_{\emptyset}\lambda$.
\end{lemma}

\begin{proof}
From Lemma \ref{lemma:decomp} we have $\overline{v}(\mu)^{-1}=\overline{v}(\mu_{I,+})^{-1}\overline{v}_{I}(\mu)^{-1}\in W^{I}W_{I}$ so $\overline{v}(\mu_{I,+})^{-1}\le_{W}\overline{v}(\mu)^{-1}$, being a subword of a reduced expression. It follows that $\overline{v}(\mu_{I,+})<_{W}\overline{v}(\lambda)$ from which $\mu_{I,+}<_{\emptyset}\lambda$.
\end{proof}

\begin{proposition}\label{prop:jacobipol}
We have $p_{I}(\lambda,k)=\sum_{w\in (W_{I})^{\lambda}}wE(\lambda,k)$ for $\lambda\in P^{+}_{I}$.
\end{proposition}
\begin{proof}
The terms that occur in the sum $\sum_{w\in (W_{I})^{\lambda}}wE(\lambda,k)$ are multiples of $e^{w\mu}$ where $\mu<_{\emptyset}\lambda$ and $w\in W_{I}$. By Lemma \ref{lemma:relativeordering}, the weights $w\mu\in P_{I}^{+}$ with this properties satisfy $w\mu<_{\emptyset}\lambda$. It follows that $\sum_{w\in (W_{I})^{\lambda}}wE(\lambda,k)=m_{I}(\lambda)+\sum_{\mu<_{I}\lambda}c_{\lambda,\mu}m_{I}(\mu)$ has the same type of expansion as $p_{I}(\lambda)$.

Let $\nu<_{I}\lambda$ and $w\in W_{I}$. If $\nu_{+}\prec\lambda_{+}$ then $w\nu<_{\emptyset}\lambda$ for all $w\in W$ and $(E(\lambda,k),e^{w\nu})_{k}=0$. If $\lambda_{+}=\nu_{+}$ then either $w\nu$ and $\lambda$ are not comparable, in which case $(E(\lambda,k),e^{w\nu})_{k}=0$, or $w\nu$ and $\lambda$ are comparable. In the latter case, suppose $\overline{v}(\lambda)\le_{W}\overline{v}(w\nu)$ or equivalently $\overline{v}(\lambda)^{-1}\le_{W}\overline{v}(w\nu)^{-1}$.
By Lemma \ref{lemma:decomp} we have $\overline{v}(w\nu)^{-1}=\overline{v}(\nu)^{-1}\overline{v}_{I}(w\nu)^{-1}$ and since the map $W=W^{I}W_{I}\to W^{I}:w''w'\mapsto w''$ is order-preserving (cf.~\cite[Prop.2.5.1]{MR2133266}) we have $\overline{v}(\lambda)^{-1}\le_{W}\overline{v}(\nu)^{-1}$, contradicting $\nu<_{\emptyset}\lambda$.
Hence $\overline{v}(\lambda)>_{W}\overline{v}(w\nu)$ and in turn $\lambda>_{\emptyset}w\nu$. We conclude that also in this case we have $(E(\lambda,k),e^{w\nu})_{k}=0$, whence $(E(\lambda,k),m_{I}(\nu))_{k}=0$ and in turn $(\sum_{w\in (W_{I})^{\lambda}}wE(\lambda,k),m_{I}(\nu))_{k}=0$.
\end{proof}



\section{Orthogonality}\label{s:ortho}

In the extreme cases $I=\emptyset$ and $I=S$ it is known that the Jacobi polynomials are pairwise orthogonal, see \cite[Cor.2.11]{MR1353018} and \cite[Cor.1.3.13]{MR1313912} respectively. For the intermediate cases we invoke the same argument. Let $\lah=\laa\otimes\bbC$. 

The graded Hecke algebra $\bfH(R^{+},k)$ is isomorphic to $S(\lah)\otimes\bbC[W]$ as a vector space and it acts on $\bbC[P]$ as described in \cite[Cor.~2.9]{MR1353018}. Most importantly, $\xi\in S(\lah)$ acts by the Cherednik operator (Def.~2.1 of loc.cit.)
$$D_{\xi}(k)=\del_{\xi}+\sum_{\alpha>0}k_{\alpha}\alpha(\xi)\frac{1}{1-e^{-\alpha}}(1-s_{\alpha})-\rho(k)(\xi)$$
which is a differential-reflection operator that is symmetric with respect to the inner product $(\cdot,\cdot)_{k}$ on $\bbC[P]$. The Cherednik operator $D_{\xi}(k)$ has $E(\lambda,k)$ as an eigenfunction,
$$D_{\xi}(k)E(\lambda,k)=\widetilde{\lambda}(\xi)E(\lambda,k),$$
where $\widetilde{\lambda}=\lambda-v(\lambda)^{-1}\rho(k)$.
Note that $\widetilde{\lambda}=\lambda+\frac{1}{2}\sum_{\alpha\in R_{+}}k_{\alpha}\eps(\lambda(\alpha^{\vee}))\alpha$ where $\eps(x)=x/|x|$ if $x\ne0$ and $\eps(0)=-1$.
The Cherednik operators $D_{\xi}(k),\xi\in\lah$ commute so we can evaluate a polynomial $q\in S(\lah)$ in the Cherednik operators to obtain $q(D(k))\in\bfH(k)$.

\begin{lemma}
If $q\in S(\lah)^{W_{I}}$ then $q(D(k))$ commutes with the action of $W_{I}$ on $\bbC[P]$. 
\end{lemma}
\begin{proof}
This follows immediately from \cite[Prop.1.1(2)]{MR1353018}.
\end{proof}

The differential-reflection operators $q(D(k))$ with $q\in S(\lah)^{W_{I}}$ leave $\bbC[P]^{W_{I}}$ invariant. The restriction of $q(D(k))$ to $\bbC[P]^{W_{I}}$ is denoted by $D_{I,q}$. Note that each $D_{I,q}$ is symmetric with respect to $(\cdot,\cdot)_{k}$.
The Jacobi polynomials $p_{I}(\lambda,k)$ are eigenfunctions for the differential-reflection operators $D_{I,q}$ with $q\in S(\lah)^{W_{I}}$,
$$D_{I,q}(p_{I}(\lambda,k))=\sum_{w\in (W_{I})^{\lambda}}D_{I,q}\left(wE(\lambda,k)\right)
=\sum_{w\in (W_{I})^{\lambda}}q(\widetilde{\lambda})wE(\lambda,k)=q(\widetilde{\lambda})p_{I}(\lambda,k).$$

To show that the polynomials $q\in S(\lah)^{W_{I}}$ separate the points $\widetilde{\lambda}$ with $\lambda\in P^{+}_{I}$ we invoke the following result. 

\begin{lemma}\label{lemma:lambdatilde}
Let $\lambda\in P^{+}_{I}$, $w^{I,\lambda}_{0}\in W_{I,\lambda}$ the longest element in $W_{I}$ that fixes $\lambda$ and $w^{I}_{0}\in W_{I}$ the longest element. Then
$\widetilde{\lambda}=w^{I,\lambda}_{0}(\lambda-w_{0}^{I}v(w_{0}^{I}\lambda)^{-1}\rho(k))$.
\end{lemma}

\begin{remark}
Lemma \ref{lemma:lambdatilde} applied to $I=\emptyset$ gives back the definition of $\widetilde{\lambda}$. Applied to $I=S$ gives $\widetilde{\lambda}=w^{\lambda}_{0}(\lambda-w_{0}v(\lambda_{-})\rho(k))=w^{\lambda}_{0}(\lambda+\rho(k))$, because $v(\lambda_{-})=e$. The latter case has been observed by Heckman, cf.~\cite[Prop.2.10]{MR1353018}.
\end{remark}

\begin{proof}
We have to show $v(\lambda)=v(w_{0}^{I}\lambda)w^{I}_{0}w^{I,\lambda}_{0}$.
Along the same lines as the proof of Lemma \ref{lemma:decomp} we have $v(w^{I}_{0}\lambda)\in W^{I}$ and $v(\lambda)=w''w'\in W^{I}W_{I}$ with $w''=v(w_{0}^{I}\lambda)$ and $w'$ the shortest element in $W_{I}$ that sends $\lambda\in P^{+}_{I}$ to $-P^{+}_{I}$, i.e.~$w'=w^{I}_{0}w^{I,\lambda}_{0}$.
\end{proof}

The description of $P^{+}_{I}$ in Proposition \ref{PplusSteinberg} was based on the Steinberg weights which were used to cover $P^{+}_{I}$ with translates of $P^{+}$. Instead, we can also cover $P$ and $P_{I}^{-}$ with translates of $P^{-}$,
\begin{equation}\label{eq:Pmin}
P=\bigcup_{w\in W}w^{-1}(\mu_{w}-P^{-}),\quad P^{-}_{I}=\bigcup_{w\in W^{I}}w^{-1}(\mu_{w}-P^{-}),
\end{equation}
where $\mu_{w}=-\sum_{\alpha\in\Pi,w^{-1}\alpha<0}\varpi_{\alpha}$. The proof is the same as that of Proposition \ref{PplusSteinberg}.
It follows that for $\lambda\in P^{-}_{I}$ we have $\lambda\in v(\lambda)^{-1}(\mu_{v(\lambda)}-P^{-})$ and $v(\lambda)\in W^{I}$. Moreover
$$\widetilde{\lambda}=\lambda-v(\lambda)^{-1}\rho(k)\in v(\lambda)^{-1}(\mu_{v(\lambda)}-P^{-})\subset P^{-}_{I},$$
from which we conclude that the map $P^{-}_{I}\to P^{-}_{I},\lambda\mapsto\widetilde{\lambda}$ is injective.
However, if $\lambda\in P^{+}_{I}$ then $\widetilde{\lambda}$ need not be in $P^{+}_{I}$ and there is \textit{a priori} no reason why $P^{+}_{I}\to P/W_{I},\lambda\to[\widetilde{\lambda}]$ would be injective.

\begin{proposition}
The polynomials in $S(\lah)^{W_{I}}$ seperate the points $\widetilde{\lambda}$ where $\lambda\in P^{+}_{I}$.
\end{proposition}
\begin{proof}
The map $P^{+}_{I}\to P^{+}_{I},\lambda\mapsto\lambda-w^{I}_{0}v(w^{I}_{0}\lambda)^{-1}\rho(k)$ is the $w^{I}_{0}$-conjugation of the injective map $P_{I}^{-}\to P_{I}^{-},\lambda\mapsto\widetilde{\lambda}$ and hence injective. 
Let $\lambda\in P^{+}_{I}$ and $q\in S(\lah)^{W_{I}}$. We have $$q(\widetilde{\lambda})=q(\lambda-w_{0}^{I}v(w_{0}^{I}\lambda)^{-1}\rho(k))$$
by Lemma \ref{lemma:lambdatilde}. Since $S(\lah)^{W_{I}}$ separates the points of $P^{+}_{I}$, it certainly separates the points of a subset.
\end{proof}

\begin{corollary}\label{cor: ortho p}
The Jacobi polynomials are pairwise orthogonal.
\end{corollary}

\begin{proof}
The polynomials $p_{I}(\lambda)$ are eigenfunctions of the invariant operators $D_{I,q},q\in S(\lah)^{W_{I}}$ with eigenvalues $q(\widetilde{\lambda})$ and for each pair $p_{I}(\lambda)\ne p_{I}(\lambda')$ with $\lambda,\lambda'\in P^{+}_{I}$ there exists a polynomial $q\in S(\lah)^{W_{I}}$ with $q(\widetilde{\lambda})\ne q(\widetilde{\lambda'})$.  
\end{proof}
%

\begin{remark}\label{remark:other gens}
In view of the description \eqref{eq:Pmin} of $P^{-}_{I}$ we obtain an alternative description of $P_{I}^{+}$, namely
$$P_{I}^{+}=\bigcup_{w\in W^{I}}w^{I}_{0} w^{-1}(\mu_{w}-P^{-}).$$
From this we obtain alternative Steinberg weights and in turn alternative generators for the $\bbC[P]^{W}$-module $\bbC[P]^{W_{I}}$. For example, in type $A_{2}$ we see that the generators from Example \ref{exampleGenerators} may be replaced by
$1,e^{\varpi_{1}-\varpi_{2}}+e^{\varpi_{2}}$ and $e^{\varpi_{1}}$.
\end{remark}

\section{Vector-valued Jacobi polynomials}\label{s:vvJacobi}
The vector space $\bbC[W/W_{I}]$ is generated by functions $\delta_{wW_{I}}:W/W_{I}\to \bbC$ (the characteristic function of $wW_{I}$ on $W/W_{I}$). We order the basis elements $\delta_{uW_{I}}$ in such away that $\delta_{uW_{I}}$ comes before $\delta_{vW_{I}}$ if $\ell(u'')<\ell(v'')$, where $u=u''u', v=v''v'\in W^{I}W_{I}$.

The space $\bbC[P]\otimes\bbC[W/W_{I}]$ of $\bbC[W/W_{I}]$-valued Laurent polynomials caries the diagonal $W$-action $w(\phi\otimes\delta_{uW_{I}})=(w\phi)\otimes\delta_{wuW_{I}}$. The space of $W$-invariants for this action,
$$\left(\bbC[P]\otimes\bbC[W/W_{I}]\right)^{W},$$
is a module over $\bbC[P]^{W}$ by $\chi(\sum \phi_{u}\otimes\delta_{uW_{I}})=\sum (\chi \phi_{u})\otimes\delta_{uW_{I}}$ for $\chi\in\bbC[P]^{W}$.

\begin{lemma}\label{lemma:GammaIso}
The map $$\Gamma:\bbC[P]^{W_{I}}\to\left(\bbC[P]\otimes\bbC[W/W_{I}]\right)^{W},\Gamma(\phi)=\sum_{u\in W^{I}}u(\phi\otimes\delta_{eW_{I}})$$
is an isomorphism of $\bbC[P]^{W}$-modules.
\end{lemma}

\begin{proof}
Note that $\Gamma(\phi)$ is $W$-invariant because the sum can be taken over $W=W^{I}W_{I}$. The first component of $\Gamma(\phi)$ is $\phi$, so $\Gamma$ is injective. If $\Phi=\sum \phi_{u}\otimes\delta_{uW_{I}}\in \left(\bbC[P]\otimes\bbC[W/W']\right)^{W}$, then $\phi_{e}$ is $W_{I}$-invariant, hence $\Phi=\Gamma(\phi_{e})$. We conclude that $\Gamma$ is a vector space isomorphism and it clearly is also an isomorphism of $\bbC[P]^{W}$-modules.
\end{proof}

The space $\bbC[P]\otimes\bbC[W/W_{I}]$ is also endowed with a Hermitean inner product, denoted by $(\cdot,\cdot)_{I,k}$ for $k_{\alpha}\ge0$, defined by
$$(\Phi,\Psi)_{I,k}=\int_{T}\tr(\overline{\Phi(t)}\Psi(t))|\delta_{k}(t)|dt,\quad\Phi,\Psi\in\bbC[P]\otimes\bbC[W/W_{I}],$$
where $\tr(\sum_{v\in W^{I}}\Phi_{v}\otimes\delta_{vW_{I}})=\sum_{v\in W^{I}}\Phi_{v}$ is just the sum of the components. Note that $(\cdot,\cdot)_{I,k}$ is $W$-invariant.
If $\Phi,\Psi\in \left(\bbC[P]\otimes\bbC[W/W_{I}]\right)^{W}$, then $\Phi=\Gamma(\phi)$ and $\Psi=\Gamma(\psi)$ for some $\phi,\psi\in\bbC[P]^{W_{I}}$ and
$$\tr(\overline{\Phi}\Psi)|\delta_{k}|=\sum_{v\in W^{I}}v\left(\overline{\phi}\psi\right)|\delta_{k}|$$
is $W$-invariant because $|\delta_{k}|$ is $W$-invariant and $\overline{\phi}\psi$ is $W_{I}$-invariant. Hence
$(\Phi,\Psi)_{I,k}=|W^{I}|(\phi,\psi)_{k}$ and we see that $\Gamma$ respects the orthogonality.
Let $P_{I}(\lambda,k)=\Gamma(p_{I}(\lambda,k))$.

\begin{proposition}
The family $\{P_{I}(\lambda,k)\mid\lambda\in P_{I}^{+}\}$ is an orthogonal basis of the vector space $\left(\bbC[P]\otimes\bbC[W/W_{I}]\right)^{W}$. 
\end{proposition}
\begin{proof}
This follows immediately from Corollary \ref{cor: ortho p} and Lemma \ref{lemma:GammaIso}.
\end{proof}

The action of a differential-reflection operator $D_{I,q}$ ($q\in S(\lah)^{W_{I}}$) on $\bbC[P]^{W_{I}}$ is transferred to an action on $\bbC[P]\otimes\bbC[W/W_{I}]$ by the the push-forward $\gamma=\Gamma_{*}$,
\begin{equation}\label{gamma}
\gamma(D_{I,q})=\sum_{u\in W^{I}}u\circ D_{I,q}\circ u^{-1}\otimes\delta_{uW_{I}}.
\end{equation}
We want to show that the restriction of $\gamma(D_{I,q})$ to $(\bbC[P]\otimes\bbC[W/W_{I}])^{W}$ is the restriction of a differential operator on $\bbC[P]\otimes\bbC[W/W_{I}]$ to the space of $W$-invariants. Note that $\gamma(D_{I,q})$ is not well defined since it depends on the choice of coset representatives in $W/W_{I}$. However, the restriction to $(\bbC[P]\otimes\bbC[W/W_{I}])^{W}$ is well defined.

Following \cite[Ch.1]{MR1313912}, let $\calR=\calR(R)$ be the unital algebra generated by $\frac{1}{1-e^{-\alpha}}, \alpha\in R_{+}$ and
$$\bbD_{I,\calR}=\calR\otimes\End(\bbC[W/W_{I}])\otimes S(\lah)$$
the algebra of differential operators on $\bbC[P]\otimes\bbC[W/W_{I}]$ with coefficients in the algebra $\calR\otimes\End(\bbC[W/W_{I}])$. Denote by $\pi_{I}:W\to\GL(\bbC[W/W_{I}])$ the representation of $W$ on $\bbC[W/W_{I}]$.
The action of $W$ on $\bbD_{I,\calR}$ is given by $w(f\otimes A\otimes \partial_{q})=(wf)\otimes(\pi_{I}(w)A)\otimes\partial_{wq}$ on simple tensors. Define
$$\bbD_{I,\calR}^{W}=\{D\in\bbD_{I,\calR}\mid w(D)=D \mbox{ for all $w\in W$}\},$$
an algebra of $W$-invariant differential operators on the space of $\bbC[W/W_{I}]$-valued Laurent polynomials. 
Let
$$\bbD R_{I,\calR}=\bbD_{I,\calR}\otimes\bbC[W]$$
and define for $D,E\in\bbD_{I,\calR}$ and $v,w\in W$ the multiplication
$$(D\otimes w)\cdot(E\otimes v)=(D\circ w(E))\otimes wv,$$
which gives $\bbD R_{I,\calR}$ the structure of an associative algebra.
The action of an element $D\otimes w\in \bbD R_{I,\calR}$ on $\Phi\in\bbC[P]\otimes\bbC[W/W_{I}]$ is given by $D(w\Phi)$.
For an element $P\in\bbD R_{I,\calR}$ of the form $P=\sum_{w\in W}P_{w}\otimes w$ we define the linear map $\beta:\bbD R_{I,\calR}\to \bbD_{I,\calR}$ by $\beta(P)=\sum_{w\in W}P_{w}$. Finally, define by $\bbD R_{I,\calR}^{1\otimes\bbC[W]}$ the commutator of $1\otimes\bbC[W]$ in $\bbD R_{I,\calR}$.
\begin{lemma}
The restriction $\beta:\bbD R_{I,\calR}^{1\otimes\bbC[W]}\to\bbD_{I,\calR}$ is a homomorphism with image in $\bbD_{I,\calR}^{W}$.
\end{lemma}

\begin{proof}
The statement for $I=S$ is shown in \cite[Lem.1.2.1,Lem.1.2.2]{MR1313912} and the general case follows \textit{mutatis mutandis} from these considerations. 
\end{proof}

The homomorphism \eqref{gamma} maps $S'(\lah)^{W_{I}}$ into $\bbD R_{I,\calR}^{1\otimes\bbC[W]}$.
We obtain an algebra homomorphism $\beta\circ\gamma:S'(\lah)^{W_{I}}\to\bbD_{I,\calR}^{W}$. The image of $\beta\circ\gamma$ is the algebra of differential operators restricted to $(\bbC[P]\otimes\bbC[W/W_{I}])^{W}$ and we denote it by
$$\bbD(I,k)=\beta(\gamma(S'(\lah)^{W_{I}})).$$
Note that
$$\beta(\gamma(D_{I,q})))(\Gamma(\phi))=\gamma(D_{I,q}))\Gamma(\phi)=\Gamma(D_{I,q}(\phi))$$ for $\phi\in\bbC[P]^{W_{I}}$, from which it follows that $\beta\circ\gamma$ is injective.
\begin{lemma}
The algebra $\bbD(I,k)$ is a commutative algebra of differential operators on $\left(\bbC[P]\otimes\bbC[W/W_{I}]\right)^{W}$. The basis elements $P(\lambda,k)$ are determined uniquely (up to scaling) as simultaneous eigenfunctions of $\bbD(I,k)$.
\end{lemma}

\begin{example}\label{example:shimenoPrep}
Consider $\bbR^{3}$ with the standard basis $(e_{1},e_{2},e_{3})$ and with the standard inner product. Let $(\eps_{1},\eps_{2},\eps_{3})$ be the basis of $(\bbR^{3})^{*}$ dual to the standard basis of $\bbR^{3}$. The orthocomplement of $e_{1}+e_{2}+e_{3}$ is denoted by $\laa$. 
Let $\Sigma=\{\pm(\eps_{1}-\eps_{2}),\pm(\eps_{1}-\eps_{3}),\pm(\eps_{2}-\eps_{3})\}$ be the root system of type $A_{2}$. The roots are denoted by $\alpha_{ij}=\eps_{i}-\eps_{j}$ and the reflections by $s_{ij}$. Note that $W=S_{3}$, the permutation group of $\{1,2,3\}$. We fix the parabolic subgroup $W'=<s_{23}>$. We identify the $W$-module $\bbC[W/W']$ with $\bbC^{3}$ on which $W$ acts by permuting the standard basis vectors.    

Let $\xi_{1}=\frac{1}{3}(2,-1,-1), \xi_{2}=\frac{1}{3}(-1,2,-1)$ and $\xi_{3}=\frac{1}{3}(-1,-1,2)$ be elements of $\lah=\laa\otimes\bbC$. Then
$$S(\lah)^{W'}=\bbC[\xi_{1},\xi_{1}\xi_{2}+\xi_{1}\xi_{3}+\xi_{2}\xi_{3}]$$
with $\xi_{1}$ and $\xi_{1}\xi_{2}+\xi_{1}\xi_{3}+\xi_{2}\xi_{3}$ algebraically independent. Our aim is to calculate the differential operators acting on the vector-valued Laurent polynomials that correspond to these two generators. In view of the relation with the spherical functions in Section \ref{s:AppI} it is convenient to consider the root system $R=2\Sigma$ with $R_{+}=\{2\alpha_{12},2\alpha_{13},2\alpha_{23}\}$ the set of positive roots and $\Pi=\{2\alpha_{12},2\alpha_{23}\}$ the set of simple roots. Accordingly $\Sigma_{+}=\frac{1}{2}R_{+}$.

We proceed to calculate $\beta(\gamma(\xi_{1}))$ and $\beta(\gamma(\xi_{1}\xi_{2}+\xi_{1}\xi_{3}+\xi_{2}\xi_{3}))$.
Let $p_{1}\in\bbC[2P]^{W'}$ and $p_{2}=s_{12}p_{1},p_{3}=s_{13}p_{1}$. Then $(p_{1},p_{2},p_{3})\in\left(\bbC[P]\otimes\bbC^{3}\right)^{W}$ and 
$$D_{\xi_{1}}(k)(p_{1})=\del_{\xi_{1}}p_{1}+2k\left(\frac{1}{1-e^{-2\alpha_{12}}}(p_{1}-p_{2})+\frac{1}{1-e^{-2\alpha_{13}}}(p_{1}-p_{3})\right)-2kp_{1}.$$
It is now easy to obtain expressions for $s_{12}D_{\xi_{1}}(k)(p_{1}),s_{13}D_{\xi_{1}}(k)(p_{1})$ and we get
\begin{multline*}
\beta(\gamma(\xi_{1}))=
\begin{pmatrix}
\del_{\xi_{1}}&0&0\\
0&\del_{\xi_{2}}&0\\
0&0&\del_{\xi_{3}}
\end{pmatrix}+\\
k
\begin{pmatrix}
\frac{e^{\alpha_{12}}+e^{-\alpha_{12}}}{e^{\alpha_{12}}-e^{-\alpha_{12}}}+\frac{e^{\alpha_{13}}+e^{-\alpha_{13}}}{e^{\alpha_{13}}-e^{-\alpha_{13}}}&-\frac{2}{1-e^{-2\alpha_{12}}}&-\frac{2}{1-e^{-2\alpha_{13}}}\\
-\frac{2}{1-e^{2\alpha_{12}}}&-\frac{e^{\alpha_{12}}+e^{-\alpha_{12}}}{e^{\alpha_{12}}-e^{-\alpha_{12}}}+\frac{e^{\alpha_{23}}+e^{-\alpha_{23}}}{e^{\alpha_{23}}-e^{-\alpha_{23}}}&-\frac{2}{1-e^{-2\alpha_{23}}}\\
-\frac{2}{1-e^{2\alpha_{13}}}&-\frac{2}{1-e^{2\alpha_{23}}}&-\frac{e^{\alpha_{13}}+e^{-\alpha_{13}}}{e^{\alpha_{13}}-e^{-\alpha_{13}}}-\frac{1e^{\alpha_{23}}+e^{-\alpha_{23}}}{e^{\alpha_{23}}-e^{-\alpha_{23}}}
\end{pmatrix}
\end{multline*}
where we have used the identity $\frac{2}{1-e^{-2\alpha}}-1=\frac{e^{\alpha}+e^{-\alpha}}{e^{\alpha}-e^{-\alpha}}$. To calculate $\beta(\gamma(\xi_{1}\xi_{2}+\xi_{1}\xi_{3}+\xi_{2}\xi_{3}))$ we introduce one further piece of notation. Upon identifying $\laa\to\laa^{*},\xi\mapsto(\xi,\cdot)$, we denote by $\lambda\mapsto \lambda'$ the inverse of this map. In particular $\alpha_{ij}'=e_{i}-e_{j}$.
A tedious calculation along the lines of \cite[Prop.1.2.3]{MR1313912} yields
\begin{multline*}
D_{\xi_{1}\xi_{2}+\xi_{1}\xi_{3}+\xi_{2}\xi_{3}}=
\del_{\xi_{1}}\del_{\xi_{2}} + \del_{\xi_{1}}\del_{\xi_{3}} + \del_{\xi_{2}}\del_{\xi_{3}}+\\
-k\sum_{\alpha\in\Sigma_{+}}
\left\{
\frac{1+e^{-2\alpha}}{1-e^{-2\alpha}}\del_{\alpha'}
\right\}+
2k\sum_{\alpha\in\Sigma_{+}}
\left\{
\frac{(\alpha,\alpha)e^{-2\alpha}}{(1-e^{-2\alpha})^{2}}\circ(1-r_{\alpha})
\right\}+\\
-2k^2\sum_{\alpha,\beta\in\Sigma_{+}}(\alpha,\beta)\left\{
\frac{1}{1-e^{-2\alpha}}(1-r_{\alpha})\frac{1}{1-e^{-2\beta}}(1-r_{\beta})\right\}+\\
2k^2\sum_{\alpha,\beta\in\Sigma_{+}}\left\{\frac{(\alpha,\beta)}{1-e^{-2\alpha}}(1-r_{\alpha})\right\}-\frac{1}{2}k^2\sum_{\alpha,\beta\in\Sigma_{+}}(\alpha,\beta).
\end{multline*}
We have $-\frac{1}{2}k^2\sum_{\alpha,\beta\in\Sigma_{+}}(\alpha,\beta)=-\frac{1}{2}k^2(2\alpha_{13},2\alpha_{13})=-4k^2$. We proceed to show that the third and second to last terms together act as zero on $\bbC[P]^{W'}$. To this end we use the identity
$$\frac{1}{1-e^{-2\alpha}}(1-r_{\alpha})\frac{1}{1-e^{-2\beta}}(1-r_{\beta})=\frac{1}{1-e^{-2\alpha}}\left(
\frac{1}{1-e^{-2\beta}}(1-r_{\alpha})-\frac{1}{1-e^{-r_{\alpha}(2\beta)}}(r_{\alpha}-r_{\alpha}r_{\beta})
\right)$$
and the effect of the simple reflections on $p\in\bbC[P]^{W'}$ given by
$$s_{12}s_{13}p=s_{13}p,\quad s_{13}s_{12}p=s_{12}p,\quad s_{23}r_{12}p=s_{13}p,\quad s_{23}s_{13}p=s_{12}p.$$
Denote $\Delta_{\alpha}=\frac{1}{1-e^{-2\alpha}}(1-s_{\alpha})$.
In Table \ref{Table:terms} we have collected the restrictions of the terms in
\begin{equation}\label{eq:third3last}
-2k^2\sum_{\alpha,\beta\in\Sigma_{+}}(\alpha,\beta)\left\{
\frac{1}{1-e^{-2\alpha}}(1-r_{\alpha})\frac{1}{1-e^{-2\beta}}(1-r_{\beta})\right\}.
\end{equation}
\begin{table}[h]
\centering
\begin{tabular}{|c|c|c|c|}
\hline
$\alpha$ & $\beta$ & $r_{\alpha}(\beta)$ & $(\alpha,\beta)\Delta_{\alpha}\Delta_{\beta}$ \\
\hline
$\alpha_{12}$ & $\alpha_{12}$ & $-\alpha_{12}$ & 
$2\left(\frac{1}{1-e^{-2\alpha_{12}}}\frac{1}{1-e^{-2\alpha_{12}}}(1-s_{12})-
\frac{1}{1-e^{-2\alpha_{12}}}\frac{1}{1-e^{2\alpha_{12}}}(s_{12}-1)\right)$\\
$\alpha_{13}$ & $\alpha_{12}$ & $-\alpha_{23}$ & 
$\left(\frac{1}{1-e^{-2\alpha_{13}}}\frac{1}{1-e^{-2\alpha_{12}}}(1-s_{12})-
\frac{1}{1-e^{-2\alpha_{13}}}\frac{1}{1-e^{2\alpha_{23}}}(s_{13}-s_{12})\right)$\\
$\alpha_{23}$ & $\alpha_{12}$ & $\alpha_{13}$ & 
$-\left(\frac{1}{1-e^{-2\alpha_{23}}}\frac{1}{1-e^{-2\alpha_{12}}}(1-s_{12})-
\frac{1}{1-e^{-2\alpha_{23}}}\frac{1}{1-e^{-2\alpha_{13}}}(1-s_{13})\right)$\\
$\alpha_{12}$ & $\alpha_{13}$ & $\alpha_{23}$ & 
$\left(\frac{1}{1-e^{-2\alpha_{12}}}\frac{1}{1-e^{-2\alpha_{13}}}(1-s_{13})-
\frac{1}{1-e^{-2\alpha_{12}}}\frac{1}{1-e^{-2\alpha_{23}}}(s_{12}-s_{13})\right)$\\
$\alpha_{13}$ & $\alpha_{13}$ & $-\alpha_{13}$ & 
$2\left(\frac{1}{1-e^{-2\alpha_{13}}}\frac{1}{1-e^{-2\alpha_{13}}}(1-s_{13})-
\frac{1}{1-e^{-2\alpha_{13}}}\frac{1}{1-e^{2\alpha_{13}}}(s_{13}-1)\right)$\\
$\alpha_{23}$ & $\alpha_{13}$ & $\alpha_{12}$ & 
$\left(\frac{1}{1-e^{-2\alpha_{23}}}\frac{1}{1-e^{-2\alpha_{13}}}(1-s_{13})-
\frac{1}{1-e^{-2\alpha_{23}}}\frac{1}{1-e^{-2\alpha_{12}}}(1-s_{12})\right)$\\
\hline
\end{tabular}
\caption{Restriction to $\bbC[P]^{W'}$ of the terms of \eqref{eq:third3last}}\label{Table:terms}
\end{table}
We also have
\begin{equation}\label{eq:third2last}
2k^2\sum_{\alpha,\beta\in\Sigma_{+}}\left\{\frac{(\alpha,\beta)}{1-e^{-2\alpha}}(1-r_{\alpha})\right\}=
2k^2\left\{\frac{2}{1-e^{-2\alpha_{12}}}(1-r_{\alpha_{12}})+\frac{4}{1-e^{-2\alpha_{13}}}(1-r_{\alpha_{13}})\right\}
\end{equation}
upon restriction to $\bbC[P]^{W'}$.
To calculate the sum of \eqref{eq:third3last} and \eqref{eq:third2last} we gather the coefficients of the group elements $1,s_{12},s_{13}$ in \eqref{eq:third2last} and Table \ref{Table:terms} and with careful bookkeeping we find that all coefficients are zero.
It follows that
\begin{multline*}
\beta(\gamma(\xi_{1}\xi_{2}+\xi_{1}\xi_{3}+\xi_{2}\xi_{3}))=\del_{\xi_{1}}\del_{\xi_{2}}+\del_{\xi_{1}}\del_{\xi_{3}}+\del_{\xi_{2}}\del_{\xi_{3}}-k\sum_{\alpha\in\Sigma_{+}}
\left\{
\frac{e^{\alpha}+e^{-\alpha}}{e^{\alpha}-e^{-\alpha}}\del_{\alpha'}
\right\}-4k^2\\
+4k
\begin{pmatrix}
\frac{1}{(e^{\alpha_{12}}-e^{-\alpha_{12}})^2}+\frac{1}{(e^{\alpha_{13}}-e^{-\alpha_{13}})^{2}}&
-\frac{1}{(1-e^{-2\alpha_{12}})^{2}}&
-\frac{1}{(1-e^{-2\alpha_{13}})^{2}}\\
-\frac{1}{(1-e^{2\alpha_{12}})^{2}}&
\frac{1}{(e^{\alpha_{12}}-e^{-\alpha_{12}})^{2}}+\frac{1}{(e^{\alpha_{23}}-e^{-\alpha_{23}})^{2}}&
-\frac{1}{(1-e^{-2\alpha_{23}})^{2}}\\
-\frac{1}{(1-e^{2\alpha_{13}})^{2}}&-\frac{1}{(1-e^{2\alpha_{23}})^{2}}&\frac{1}{(e^{\alpha_{13}}-e^{-\alpha_{13}})^{2}}+\frac{1}{(e^{\alpha_{23}}-e^{-\alpha_{23}})^{2}}
\end{pmatrix}.
\end{multline*}

\end{example}

\section{Application I}\label{s:AppI}
Consider the compact symmetric pair $(U,K)$ equal to one of the following three, 
$$(\SU(3),\SO(3)),\quad (\SU(3)\times\SU(3),\mathrm{diag}(\SU(3))),\quad(\SU(6),\Sp(6)).$$
The restricted root system $\Sigma$ of each of these pairs is of type $A_{2}$ and the dimensions of the root spaces are $m=1,2,4$ respectively. Following the structure theory of compact symmetric spaces, $\Sigma\subset\laa^{*}$ where $i\laa=\mathrm{Lie}(A)$ for a compact torus $A\subset U$ with the property that $U=KAK$. We denote $M=Z_{K}(A)$, the centralizer of $A$. We identify the complex valued regular functions on $A$ with $\bbC[P_{\Sigma}]$.
In describing irreducible representations of $U$ and $K$ we use the theorem of the highest weight.

Let $\pi^{K}_{\mu}:K\to\GL(V^{K}_{\mu})$ be an irreducible $K$-representation of highest weight $\mu$. The space of regular functions $F:U\to\End(V^{K}_{\mu})$ that satisfy $F(ku\ell)=\pi(k)F(u)\pi(\ell)$ for all $k,\ell\in K, u\in U$ is denoted by $E^{\mu}$. The transformation behavior implies that each function in $E^{\mu}$ is completely determined by its restriction to $A$ so we obtain an isomorphism
$E^{\mu}\to E^{\mu}_{A}$ be restricting to $A$. Note that
$$E^{\mu}_{A}\subset \left(\bbC[P_{\Sigma}]\otimes\End(V^{K}_{\mu})\right)^{N_{K}(A)}=\left(\bbC[P_{\Sigma}]\otimes\End_{M}(V^{K}_{\mu})\right)^{(M\cap A)\cdot W_{\Sigma}}.$$
For example, if $\mu=0$ then $E^{0}_{A}=\bbC[P_{\Sigma}]^{N_{K}(A)}=\bbC[2P_{\Sigma}]^{W_{\Sigma}}$. In this special case it is known that $\pi^{K}_{0}$, the trivial $K$-representation, occurs in every irreducible $U$-representation with multiplicity at most one. In fact, using the Cartan-Helgason theorem, an irreducible $U$-representation $\pi^{U}_{\lambda}$ admits a $K$-invariant subspace precisely when $\lambda\in 2P_{\Sigma}$ is dominant. It implies, using the Peter-Weyl theorem, that $\bbC[2P_{\Sigma}]^{W_{\Sigma}}$ can be equipped with an orthogonal basis whose members, the restricted zonal spherical functions, are uniquely determined (up to scaling) as simultaneous eigenfunctions of a commutative subquotient of the universal enveloping algebra $U(\lau_{\bbC})$ of the complexification $\lau_{\bbC}$ of $\lau=\mathrm{Lie}(U)$. This commutative subquotient can be identified with $S(\laa_{\bbC})^{W_{\Sigma}}$ by the Harish-Chandra isomorphism, cf.~\cite[\S5]{MR1313912}.
To describe the restricted zonal spherical functions for our symmetric pairs in terms of Heckman-Opdam polynomials we take $R=2\Sigma$ and $k=\frac{m}{2}$ for the root multiplicity.
We use $k=1/2,1,2$ for reference to the symmetric pairs and denote the corresponding groups by $U_{k},K_{k}, M_{k},A_{k}$.

For each of the symmetric pairs we consider the spherical functions for the defining representation $\pi^{K}_{\mu_{1}}$ of $K$, where $\mu_{1}$ is the first fundamental weight of $K$. The induction of $\pi^{K}_{\mu_{1}}$ to $U$ is multiplicity free, i.e.~we have $[\pi^{U}_{\lambda}|_{K}:\pi^{K}_{\mu_{1}}]\le1$ for each dominant weight of $U$. Indeed, for $k=1,2$ the triples $(U,K,\mu_{1})$ are multiplicity free triples that are obtained from multiplicity free systems in the sense of \cite[\S8]{PvP}. This is not the case for $k=1/2$, but the multiplicity freeness is guaranteed by Lemma \ref{lemma:branching1/2} below. 
For $\pi^{U}_{\lambda}$ that contains $\pi^{K}_{\mu_{1}}$ we fix $K$-equivariant linear maps $j^{\mu_{1}}_{\lambda}:V^{K}_{\mu_{1}}\to V^{U}_{\lambda}$ and $p^{\lambda}_{\mu_{1}}:V^{U}_{\lambda}\to V^{K}_{\mu}$ for which $p^{\lambda}_{\mu_{1}}j^{\mu_{1}}_{\lambda}$ is the identity.
We define the spherical function $\Phi^{\mu_{1}}_{\lambda}\in E^{\mu_{1}}$ of type $\mu_{1}$ associated to $\lambda$ by
$$\Phi^{\mu_{1}}_{\lambda}(u)=p^{\lambda}_{\mu_{1}}\pi^{U}_{\lambda}(u)j^{\mu_{1}}_{\lambda}.$$
For $k=1,2$ it is known that $E_{A}^{\mu_{1}}$ is a freely and finitely generated $E^{0}_{A}$-module generated by three spherical functions, see \cite[Thm.8.12]{PvP}. The proof is based on the fact that the set $P_{U}^{+}(\mu_{1})$ of dominant $U$-weights $\lambda$ for which $[\pi^{U}_{\lambda}|_{K}:\pi^{K}_{\mu_{1}}]=1$ is of the form
$$P^{+}_{U}(\mu_{1})\cong\{b_{1},b_{2},b_{3}\}+P_{U}^{+}(0)$$
where $b_{1},b_{2},b_{3}$ are elements of $P_{U}^{+}$. The generators are then the spherical functions associated to $b_{1},b_{2}$ and $b_{3}$. The set $P_{U}^{+}(0)$ is a monoid whose elements are called spherical weights. The fundamental weights are denoted by $\omega_{i}$ where $i$ ranges appropriately.

\begin{lemma}\label{lemma:branching1/2}
If $k=\frac{1}{2}$ then the set of dominant weights $\lambda$ of $\SU(3)$ for which $[\pi^{U_{1/2}}_{\lambda}|_{K}:\pi^{K}_{\mu_{1}}]\ge1$ is given by
$$P^{+}_{U_{1/2}}(\mu_{1})=\{\omega_{1},\omega_{1}+\omega_{2},\omega_{2}\}+2P_{U_{1/2}}^{+},$$
where $2P_{U_{1/2}}^{+}$ is the set of spherical weights. Moreover, for $\lambda\in P_{+}^{U_{1/2}}(\mu_{1})$ we have $[\pi^{U_{1/2}}_{\lambda}|_{K_{1/2}}:\pi^{K_{1/2}}_{\mu_{1}}]=1$.  
\end{lemma}
\begin{proof}
Here $M_{1/2}$ consists of the diagonal matrices with $\pm1$ on the diagonal and determinant one. From an argument similar to that of the proof of  \cite[Thm.3.1]{MR3801483} we see that $[\pi^{U_{1/2}}_{\lambda}|_{K_{1/2}}:\pi^{K_{1/2}}_{\omega_{1}}]$ is bounded by $[\pi^{K_{1/2}}_{\mu_{1}}|_{M_{1/2}}:\pi^{M_{1/2}}_{\lambda_{*}}]$ which is at most one.
Moreover, that the restriction of $\lambda\in P^{+}_{U_{1/2}}(\mu_{1})$ to $M_{1/2}$ is a character that occurs as a constituent of $V^{K_{1/2}}_{\mu_{1}}$ which can be described by $\omega_{1},\omega_{2}-\omega_{1},-\omega_{2}$, where we interpret for a moment the weights as characters of the torus. The restrictions of two weights are the same if they differ by a weight of $2P_{U_{1/2}}$. It follows that $P_{U_{1/2}}^{+}(\omega_{1})\subset \{\omega_{1},\omega_{1}+\omega_{2},\omega_{2}\}+2P^{+}_{U_{1/2}}$. To show equality we have to show that $\pi^{U_{1/2}}_{\omega_{1}},\pi^{U_{1/2}}_{\omega_{1}+\omega_{2}},\pi^{U_{1/2}}_{\omega_{2}}$ all contain $\pi^{K}_{\mu_{1}}$ upon restriction. Only for $\omega_{1}+\omega_{2}$ there is something to show. Since $V^{U}_{\omega_{1}+\omega_{2}}\subset \End(\bbC^{3})$ where the action of $\SU(3)$ is by conjugation, we observe that that the 3-dimensional space of anti-symmetric matrices invariant for $K=\SO(3)$ and is in fact isomorphic to $\bbC^{3}$ via 
\begin{equation}\label{embeddingSkew}
(a,b,c)\mapsto\begin{pmatrix}
0&-c&b\\
c&0&-a\\
-b&a&0
\end{pmatrix}
\end{equation}
as an $\SO(3)$-module.
\end{proof}

The spherical functions restricted to $A$ take their values in $\End_{M}(V^{K}_{\mu_{1}})$ which can be identified with $\bbC^{3}$, as follows. In each case $k=1/2,1,2$ we fix a basis of $V^{K}_{\mu_{1}}$ that contains at the same time a basis for the three irreducible submodules for $M$. Each of the three blocks of $\Phi^{\lambda}_{\mu_{1}}(a)\in\End_{M}(V^{K}_{\mu_{1}})$ is a multiple of the identity and we send this multiple to the corresponding entry in $\bbC^{3}$ following the ordering of the bases below.  
\begin{itemize}
\item $V^{K_{1/2}}_{\mu_{1}}\cong\bbC^{3}$ and the standard basis elements $e_{1},e_{2},e_{3}$ transforms under $M_{1/2}$ via $\omega_{1},\omega_{1}+\omega_{2},\omega_{2}$, viewed as characters on the group for the moment. The group $A_{1/2}$ consists of the diagonal matrices $\diag(t_{1},t_{2},t_{3})\in\SU(3)$.
\item $V^{K}_{\mu_{1}}=\bbC^{3}$ and the standard basis $\{e_{1},e_{2},e_{3}\}$ is a weight basis for $M_{1}\subset\textrm{diag}(\SU(3))$, the maximal diagonal torus. The group $A_{1}$ is the anti-diagonal embedding $a\mapsto(a,a^{-1})$ of  the group $A_{1/2}$ in $\SU(3)\times\SU(3)$. 
\item $V^{K}_{\mu_{1}}=\bbC^{6}$ and $M_{2}=\SU(2)\times\SU(2)\times\SU(2)$. The standard basis $\{e_{1},\ldots,e_{6}\}$ has sub-bases $\{e_{1},e_{4}\},\{e_{2},e_{5}\},\{e_{3},e_{6}\}$ for the irreducible $M_{2}$-modules in the decomposition $\bbC^{6}=\bbC^{2}\times\bbC^{2}\times\bbC^{2}$. The group $A_{2}$ consist of diagonal matrices $\diag(t_{1},t_{2},t_{3},t_{1},t_{2},t_{3})$ for which $\diag(t_{1},t_{2},t_{3})\in A_{1/2}$.
\end{itemize}
The groups $A_{k}$ are all isomorphic to $A_{1/2}$ and $M_{k}\cap A_{k}$ is isomorphic to $M_{1/2}$ under this identification. The induced action of $W_{\Sigma}\cong S_{3}$ on $\bbC^{3}$ is given by permutation of the standard basis elements. 
In this way we identify each of the spaces $E^{\mu_{1}}_{A_{k}}$ with a subspace $\calE_{k}$ of 
$$\calE=\left(\bbC[P_{\Sigma}]\otimes\bbC^{3}\right)^{(M_{1/2}\cdot W_{\Sigma})},$$
where $\Sigma$ is the restricted root system for $(U_{1/2},K_{1/2})$ and $P_{\Sigma}$ is the weight lattice for this root system.
A restricted spherical function $\Phi^{\mu_{1}}_{\lambda}\in E^{\mu_{1}}_{A_{k}}$ defines a unique function $\Psi^{\mu_{1}}_{\lambda}(k)\in\calE$.
Upon writing $\calE^{0}=\bbC[2P_{\Sigma}]^{W_{\Sigma}}$ we see that $\calE,\calE_{k},k=1/2,1,2$ are $\calE_{0}$-modules.

\begin{proposition}\label{prop:gen}
The $\calE^{0}$-modules $\calE_{k},k=1/2,1,2$ are equal and as an $\calE^{0}$-module it is freely generated by the functions
$$\Psi^{\mu_{1}}_{\omega_{1}}=
\begin{pmatrix}
e^{\omega_{1}}\\e^{\omega_{2}-\omega_{1}}\\e^{-\omega_{2}}
\end{pmatrix},\quad
\Psi^{\mu_{1}}_{\omega_{1}+\omega_{2}}=\frac{1}{2}
\begin{pmatrix}
e^{\alpha_{2}}+e^{-\alpha_{2}}\\e^{\alpha_{3}}+e^{-\alpha_{3}}\\e^{\alpha_{1}}+e^{-\alpha_{1}}
\end{pmatrix},\quad
\Psi^{\mu_{1}}_{\omega_{2}}=
\begin{pmatrix}
e^{-\omega_{1}}\\e^{\omega_{1}-\omega_{2}}\\e^{\omega_{2}}
\end{pmatrix}.$$
\end{proposition}

\begin{proof}
The space $E^{\mu_{1}}_{A_{1/2}}$ is generated as a module over $E^{0}_{A_{1/2}}$ by the spherical functions $\Phi^{\mu_{1}}_{b}$ with $b\in\{\omega_{1},\omega_{1}+\omega_{2},\omega_{2}\}$, which follows from the proof of \cite[Thm.8.12]{PvP}.
The space $E^{\mu_{1}}_{A_{k}}$ with $k=1,2$, is freely generated as a module over $E^{0}_{A}$ by the spherical functions $\Phi^{\mu_{1}}_{b}$ with $b\in\{(\omega_{1},0),(\omega_{2},\omega_{1}),(0,\omega_{1})\}$ for $k=1$ and $b\in\{\omega_{1},\omega_{3},\omega_{5}\}$, see \cite[Case B.1.1]{PvP}.
Since $E^{0}_{A_{k}}=\bbC[2P_{\Sigma}]^{W_{\Sigma}}$ is the same for $k=1/2,1,2$, we only have to calculate the spherical functions for the weights $b_{1},b_{2},b_{3}$ in the three cases.

For $k=1$ the spherical functions have been calculated in \cite[\S8.1]{MR4053617}. If $k=1/2$, then the spherical functions $\Phi^{\mu_{1}}_{b}$ are easily calculated for $b=\omega_{1},\omega_{2}$. For $b=\omega_{1}+\omega_{2}$ use the embedding $\bbC^{3}\to\End(\bbC^{3})$ given by \eqref{embeddingSkew}.
Likewise, for $k=2$, the calculation of $\Phi^{\mu_{1}}_{b}$ with $b=\omega_{1},\omega_{5}$ is straightforward. For $b=\omega_{3}$ note that
$$\bbC^{6}\to\bigwedge^{3}\bbC^{6}, e_{1}\mapsto e_{1}\wedge e_{2}\wedge e_{5}+e_{1}\wedge e_{3}\wedge e_{6}$$
induces an $\Sp(6)$-equivariant embedding. The matrix coefficient for this vector with itself, restricted to $A_{2}$, is given by $t_{2}t_{3}^{-1}+t_{3}t_{2}^{-1}$ and it is the first entry of $\Phi^{\mu_{1}}_{\omega_{3}}$. The other entries are determined by the action of $W_{\Sigma}$.
\end{proof}

The inner product for spherical functions on the group is given by integration over the compact group $U$ and it can be reduced to integration over $A$, cf.~\cite[(2.6)]{MR4053617}.
In this way the space $E^{\mu_{1}}_{A}$ carries an inner product that is given by
$$\langle\Phi,\Psi\rangle=\int_{A}\Phi(a)^{*}\Psi(a)\delta_{k}(a)da$$
with $\delta_{k}$ the same as \eqref{eq:innerproduct}. 
We equip $\calE$ with the inner product in such a way that the isomorphisms $E^{\mu_{1}}_{A,k}\to\calE_{k}$ become unitary for $k=1/2,1$. For $k=2$ the identification is unitary up to a factor two.

\begin{lemma}\label{lemma:T}
The map $\calE\to(\bbC[2P]\otimes\bbC^{3})^{S_{3}}$
given by point wise multiplication by
$$T=\begin{pmatrix}
e^{\omega_{1}}&0&0\\
0&e^{\omega_{2}-\omega_{1}}&0\\
0&0&e^{-\omega_{2}}
\end{pmatrix}$$
is a unitary isomorphism of $\bbC[2P]^{S_{3}}$-modules.
\end{lemma}

\begin{proof}
The matrix $T$ is $S_{3}$-invariant and it absorbs the $M_{1/2}$-action, i.e~ $T\Phi\in(\bbC[P]^{M_{1/2}}\otimes\bbC^{3})^{S_{3}}$. At the same time $\bbC[P]^{M_{1/2}}=\bbC[2P]$ from which the claim follows.
\end{proof}

Note that the generators of the $\calE^{0}$-modules $\calE_{k}$ from Proposition \ref{prop:gen} are mapped to
$$\Phi^{\mu_{1}}_{\varpi_{1}}=
\begin{pmatrix}
e^{2\omega_{1}}\\e^{2\omega_{2}-2\omega_{1}}\\e^{-2\omega_{2}}
\end{pmatrix},\quad
\Phi^{\mu_{1}}_{\varpi_{1}+\varpi_{2}}=\frac{1}{2}
\begin{pmatrix}
e^{2\omega_{1}-2\omega_{2}}+e^{2\omega_{2}}\\e^{-2\omega_{1}}+e^{2\omega_{2}}\\e^{-2\omega_{1}}+e^{2\omega_{1}-2\omega_{2}}
\end{pmatrix},\quad
\Phi^{\mu_{1}}_{\varpi_{2}}=
\begin{pmatrix}
1\\1\\1
\end{pmatrix}$$
by multiplication with $T$, which are in turn generators of the $\bbC[2P]^{S_{3}}$-module $(\bbC[2P]\otimes\bbC^{3})^{S_{3}}$ by Remark \ref{remark:other gens}.

\begin{corollary}
The $\calE^{0}$ modules $\calE$ and $\calE_{k},k=1/2,1,2$ are all the same. 
\end{corollary}

The spherical functions are determined as simultaneous eigenfunctions of a commutative subquotient of $U(\lau_{\bbC})$ that acts as an algebra of differential operators. To pass from the elements in the universal enveloping algebra to actual differential operators one uses the radial part map of Harish-Chandra, see e.g.~\cite{MR683007}. The context of loc.cit.~is of non-compact Riemann symmetric pairs, but the all calculations are of an algebraic nature. Shimeno \cite{MR3775397} has calculated the generators of this algebra of differential operators in the non-compact context for the cases that we are looking at. The interpretation in the compact case gives the generators
\begin{multline}
\calD_{1}=\begin{pmatrix}
\del_{\xi_{1}}&0&0\\
0&\del_{s_{\alpha_{1}}\xi_{1}}&0\\
0&0&\del_{s_{\alpha_{3}}\xi_{1}}
\end{pmatrix}\\
+k\cdot\begin{pmatrix}
\frac{e^{\alpha_{1}}+e^{-\alpha_{1}}}{e^{\alpha_{1}}-e^{-\alpha_{1}}}+\frac{e^{\alpha_{3}}+e^{-\alpha_{3}}}{e^{\alpha_{3}}-e^{-\alpha_{3}}} & 
-\frac{2}{e^{\alpha_{1}}-e^{-\alpha_{1}}}& 
-\frac{2}{e^{\alpha_{3}}-e^{-\alpha_{3}}}\\
\frac{2}{e^{\alpha_{1}}-e^{-\alpha_{1}}}&
-\frac{e^{\alpha_{1}}+e^{-\alpha_{1}}}{e^{\alpha_{1}}-e^{-\alpha_{1}}}+\frac{e^{\alpha_{3}}+e^{-\alpha_{3}}}{e^{\alpha_{3}}-e^{-\alpha_{3}}}& 
-\frac{2}{e^{\alpha_{2}}-e^{-\alpha_{2}}}\\
\frac{2}{e^{\alpha_{3}}-e^{-\alpha_{3}}}&
\frac{2}{e^{\alpha_{2}}-e^{-\alpha_{2}}}&
 -\frac{e^{\alpha_{3}}+e^{-\alpha_{3}}}{e^{\alpha_{3}}-e^{-\alpha_{3}}}-\frac{e^{\alpha_{2}}+e^{-\alpha_{2}}}{e^{\alpha_{2}}-e^{-\alpha_{2}}}
\end{pmatrix}
\end{multline}
and 
\begin{multline}
\calD_{2}=\del_{\xi_{1}}\del_{\xi_{2}}+\del_{\xi_{1}}\del_{\xi_{3}}+\del_{\xi_{2}}\del_{\xi_{3}}-\sum_{i<j}k\frac{1+e^{-2\alpha_{ij}}}{1-e^{-2\alpha_{ij}}}\del_{e_{i}-e_{j}}+\\
2k\cdot\begin{pmatrix}
\frac{2}{(e^{\alpha_{1}}-e^{-\alpha_{1}})^{2}}+\frac{2}{(e^{\alpha_{3}}-e^{-\alpha_{3}})^{2}} & 
-\frac{e^{\alpha_{1}}+e^{-\alpha_{1}}}{(e^{\alpha_{1}}-e^{-\alpha_{1}})^{2}}& 
-\frac{e^{\alpha_{3}}+e^{-\alpha_{3}}}{(e^{\alpha_{3}}-e^{-\alpha_{3}})^{2}}\\
-\frac{e^{\alpha_{1}}+e^{-\alpha_{1}}}{(e^{\alpha_{1}}-e^{-\alpha_{1}})^{2}}&
\frac{2}{(e^{\alpha_{1}}-e^{-\alpha_{1}})^{2}}+\frac{2}{(e^{\alpha_{3}}-e^{-\alpha_{3}})^{2}}& 
-\frac{e^{\alpha_{2}}-e^{-\alpha_{2}}}{(e^{\alpha_{2}}-e^{-\alpha_{2}})^{2}}\\
-\frac{e^{\alpha_{3}}+e^{-\alpha_{3}}}{(e^{\alpha_{3}}-e^{-\alpha_{3}})^{2}}&
-\frac{e^{\alpha_{2}}+e^{-\alpha_{2}}}{(e^{\alpha_{2}}-e^{-\alpha_{2}})^{2}}&
\frac{2}{(e^{\alpha_{3}}-e^{-\alpha_{3}})^{2}}+\frac{2}{(e^{\alpha_{2}}-e^{-\alpha_{2}})^{2}}
\end{pmatrix}.
\end{multline}

\begin{proposition}
We have $$T\circ\beta(\gamma(\xi_{1}))\circ T^{-1}=\calD_{1},$$
$$T\circ\beta(\gamma(\xi_{1}\xi_{2}+\xi_{2}\xi_{3}+\xi_{1}\xi_{3}))\circ T^{-1}=\calD_{2}-\calD_{1}-4k^2-\frac{1}{3}.$$
\end{proposition}

\begin{proof}
The differential operators $\beta(\gamma(\xi_{1}))$ and $\beta(\gamma(\xi_{1}\xi_{2}+\xi_{2}\xi_{3}+\xi_{1}\xi_{3}))$ have been calculated in Example \ref{example:shimenoPrep}. The conjugation with $T$ is a tedious calculation that is based, among other things, on the identities
\begin{multline}
e^{\omega_{1}}\circ(\del_{1}'\del_{2}'+\del_{1}'\del_{3}'+\del_{2}'\del_{3}')\circ e^{-\omega_{1}}
=\del_{1}'\del_{2}'+\del_{1}'\del_{3}'+\del_{2}'\del_{3}'+\frac{2}{3}\del_{1}'-\frac{1}{3}\del_{2}'-\frac{1}{3}\del_{3}'-\frac{1}{3}\\
=\del_{1}'\del_{2}'+\del_{1}'\del_{3}'+\del_{2}'\del_{3}'+\del_{1}'-\frac{1}{3},
\end{multline}
and
$$-4\frac{e^{-\alpha}}{(1-e^{-2\alpha})^{2}}=-2\frac{e^{\alpha}+e^{-\alpha}}{(e^{\alpha}-e^{-\alpha})^{2}}-\frac{2}{e^{\alpha}-e^{-\alpha}}$$
and is left for the reader.
\end{proof}

\begin{corollary}
The spherical functions of type $\mu_{1}$ for the compact symmetric pairs $(U,K)$ in this section can be identified with the Jacobi polynomials from Example \ref{example:shimenoPrep}, up to multiplication with the function $T$ from Lemma \ref{lemma:T} and up to scaling.
\end{corollary}

\begin{proof}
The spaces spanned by the orthogonal families of functions can be identifies and so can the commutative algebras that determine the individual members of the familes.
\end{proof}

The spherical functions are normalized to be the identity in the unit element. The norms of the spherical functions can be expressed in terms of dimensions of the underlying representation spaces, which can be expressed by the Weyl dimension formulas for irreducible $U$- and $K$-modules cf.~\cite[\S2.1]{MR4053617}. This suggests that the normalizations of the Jacobi polynomials in this paper have interesting $k$-dependencies.  

\begin{remark}
The compact symmetric pair $(U,K)$ with Dynkin types $E_{6}$ and $F_{4}$ respectively, also has a restricted root system of Dynkin type $A_{2}$. The first fundamental representation of $K$ is 52-dimensional and induces multiplicity free to $E_{6}$, see \cite[\S 9, Table 2 (B10)]{PvP}. However, the restriction of this representation to $M$, whose Lie algebra is of Dynkin type $D_{4}$, is the sum of four irreducible $M$-modules. It follows that the spherical functions for this $K$-type do not fit into this examples of families. It could be that there is a irreducible $K$-module that decomposes into three different irreducible $M$-modules, we don't know if this is the case.
\end{remark}

\section{Application II}\label{s:AppII}
To describe the $W$-invariant polynomials as genuine polynomials we follow \cite{MR447979}.
The $W$-invariant Laurent polynomials on $A$ constitute the polynomial algebra $\bbC[P]^{W}=\bbC[\chi_{1},\ldots,\chi_{n}]$, where $n=\rank R$. The image of the map $\chi:T\to\bbC^{n}:t\mapsto(\chi_{1}(t),\ldots,\chi_{n}(t))$ is denoted by $\Omega=\chi(A)$ and it is compact and contained in a totally real subspace $\bbR^{n}\subset\bbC^{n}$. For $p,q\in\bbC[x_{1},\ldots,x_{n}]$ have
$$\int_{A}\overline{p(\chi(t))}q(\chi(t))\delta_{k}(t)dt=\int_{\Omega}\overline{p(x)}q(x)w_{k}(x)dx$$
for some specific weight function $w_{k}$ for which $w_{k}(\chi)$ is the product of $\delta_{k}$ and the absolute value of the Jacobian of $\chi$. The zero set of the latter is in general contained in the zero set of the former.

We proceed to describe the space $\bbC[P]^{W_{I}}$ also as a space of genuine polynomials. Recall that we have fixed a total ordering on $W^{I}$ in which $u$ comes before $v$ if $\ell(u)\le\ell(v)$. We use this to identify $\bbC[W/W_{I}]=\bbC^{|W^{I}|}$.

The isomorphism $\bbC[P]^{W_{I}}\to\bbC[P]^{W}\otimes\bbC^{|W^{I}|}$ defined by $\sum_{v\in W^{I}}f_{v}\phi_{v}\mapsto(f_{v},v\in W^{I})$ induces an isomorphism $\left(\bbC[P]\otimes\bbC[W/W_{I}]\right)^{W}\to\bbC[P]^{W}\otimes\bbC^{|W^{I}|}$ given by
$$\sum_{v\in W^{I}}f_{v}\Gamma(\phi_{v})\mapsto(f_{v},v\in W^{I}).$$
Let $\Phi_{I}$ be $|W^{I}|\times|W^{I}|$ matrix whose columns are $\Gamma(\phi_{v}),v\in W^{I}$.
Using the identification $\bbC[W/W_{I}]=\bbC^{|W^{I}|}$, we see that the corresponding map
$$(\bbC[P]\otimes\bbC^{|W^{I}|})^{W}\to\bbC[P]^{W}\otimes\bbC^{|W^{I}|}$$
is given by point wise multiplication with the inverse of $\Phi_{I}$. Note that
 $$\det\Phi_{I}=\prod_{\alpha>0}(e^{\alpha/2}-e^{-\alpha/2})^{n_{\alpha}}$$
where $n_{\alpha}$ is the number of pairs in $W/W_{I}$ that is interchanged by $s_{\alpha}$ upon left multiplication, see \cite[Lemma 2.9]{MR372897}. Let $\bbM_{I}=\mathrm{Mat}(\bbC,|W^{I}|\times|W^{I}|)$ on which we have the usual Hermitean adjoint which we denote by $C\mapsto C^{*}$. 

Let $\calW_{I}\in\bbC[x_{1},\ldots,x_{n}]\otimes\bbM_{I}$ be defined by $\calW_{I}(\chi)=\Phi_{0}^{*}\Phi_{0}$.
Then $\calW_{I}w_{k}$ is a matrix weight, i.e.~for $Q_{1},Q_{2}\in\bbC[x_{1},\ldots,x_{n}]\otimes\bbC^{|W^{I}|}$ the pairing
\begin{equation}\label{eq:weight}
\langle Q_{1},Q_{2}\rangle_{I,k}=\int_{\Omega} Q_{1}(x)^{*}\calW_{I}(x) Q_{2}(x)w_{k}(x)dx
\end{equation}
defines an inner product. By construction the map
$$\left((\bbC[P]\otimes\bbC^{|W^{I}|})^{W},(\cdot,\cdot)_{I,k}\right)\to \left(\bbC[x_{1},\ldots,x_{n}]\otimes\bbC^{|W^{I}|},\langle\cdot,\cdot\rangle_{I,k}\right),$$
given by multiplication with $\Phi^{-1}_{I}$, is a unitary isomorphism.

Using the identifications from above we define $\calP_{I}(\lambda,k)\in\bbC[x_{1},\ldots,x_{n}]\otimes\bbC^{|W^{I}|}$ by
$$\calP_{I}(\lambda,k)(\chi)=\Phi_{I}^{-1}P_{I}(\lambda,k).$$

The family $(\calP_{I}(\lambda,k),\lambda\in P^{+}_{I})$ is an orthogonal basis of $\bbC[x_{1},\ldots,x_{n}]\otimes\bbC^{|W^{I}|}$ with respect to $\langle\cdot,\cdot\rangle_{I,k}$.
Denote by $\calD(I,k)=\{\Phi_{I}^{-1}\circ D\circ\Phi_{I}\mid D\in\bbD(I,k)\}$. Then $\calD(I,k)$ is an algebra of differential operators acting on $\bbC[x_{1},\ldots,x_{n}]\otimes\bbC^{|W^{I}|}$. The $\calP_{I}(\lambda,k),\lambda\in P^{+}_{I}$ are simultaneous eigenfunctions and are in fact separated as such, by construction.
Since $\calD(I,k)$ acts on $\bbC[x_{1},\ldots,x_{n}]\otimes\bbC^{|W^{I}|}$, we see that its coefficients are $\bbM_{I}[x_{1},\ldots,x_{n}]$-valued.

Let $\calM_{I}(\sigma)\in\bbM_{I}[x_{1},\ldots,x_{n}]$ be the $\bbM_{I}$-valued polynomial whose columns are $\calP_{I}(v^{-1}(\lambda_{v}+\sigma),k),v\in W^{I}$. Then $(\calM_{I}(\sigma,k),\sigma\in P^{+})$ is a family of matrix-valued orthogonal polynomials with respect to \eqref{eq:weight}, now interpreted as an $\bbM_{I}$-valued inner product. These polynomials are uniquely determined up to scaling by $\bbM$ from the right, as simultaneous eigenfunctions of the commutative algebra $\calD(I,k)$, the characters now taking values in subalgebra of the diagonal matrices in $\bbM_{I}$ which act on the polynomials on the right.

\begin{example}
\begin{itemize}
\item If $\rank(R)=1$ and $I=S$, then we obtain for $\calP_{I}(\lambda,k)$ the classical Jacobi polynomials in a single variable.
\item If $R=A_{2}$ and $I=\{s_{2}\}$, then we obtain the the family of polynomials described in \cite[\S8.1]{MR4053617}. This is only the geometric case, i.e. $k=1$, but our theory allows that we vary the parameter $k$. In this case it is known that the matrix weight is indecomposable, i.e.~it does not reduce to smaller blocks.
\end{itemize}
\end{example}

\textbf{Acknowledgment} I would like to thank Gert Heckman and Philip Schlösser for useful remarks on earlier versions of this manuscript.

\bibliography{bibfileMVP}{}

\begin{thebibliography}{10}

\bibitem{MR2133266}
A.~Bj\"{o}rner and F.~Brenti.
\newblock {\em Combinatorics of {C}oxeter groups}, volume 231 of {\em Graduate
  Texts in Mathematics}.
\newblock Springer, New York, 2005.

\bibitem{MR683007}
W.~Casselman and D.~Mili\v{c}i\'{c}.
\newblock Asymptotic behavior of matrix coefficients of admissible
  representations.
\newblock {\em Duke Math. J.}, 49(4):869--930, 1982.

\bibitem{MR1313912}
G.~Heckman and H.~Schlichtkrull.
\newblock {\em Harmonic analysis and special functions on symmetric spaces},
  volume~16 of {\em Perspectives in Mathematics}.
\newblock Academic Press, Inc., San Diego, CA, 1994.

\bibitem{MR4053617}
E.~Koelink, M.~van Pruijssen, and P.~Rom\'{a}n.
\newblock Matrix elements of irreducible representations of {${\rm
  SU}(n+1)\times {\rm SU}(n+1)$} and multivariable matrix-valued orthogonal
  polynomials.
\newblock {\em J. Funct. Anal.}, 278(7):108411, 48, 2020.

\bibitem{MR1976581}
I.~G. Macdonald.
\newblock {\em Affine {H}ecke algebras and orthogonal polynomials}, volume 157
  of {\em Cambridge Tracts in Mathematics}.
\newblock Cambridge University Press, Cambridge, 2003.

\bibitem{MR1353018}
E.M. Opdam.
\newblock Harmonic analysis for certain representations of graded {H}ecke
  algebras.
\newblock {\em Acta Math.}, 175(1):75--121, 1995.

\bibitem{PvP}
G.~Pezzini and M.~van Pruijssen.
\newblock On the extended weight monoid of a spherical homogeneous space and
  its applications to spherical functions.
\newblock accepted for publication in \textit{Representation Theory}, 2020.

\bibitem{MR3775397}
N.~Shimeno.
\newblock Matrix valued commuting differential operators with {$A_2$} symmetry.
\newblock In {\em Geometric and harmonic analysis on homogeneous spaces and
  applications}, volume 207 of {\em Springer Proc. Math. Stat.}, pages
  157--184. Springer, Cham, 2017.

\bibitem{MR372897}
R.~Steinberg.
\newblock On a theorem of {P}ittie.
\newblock {\em Topology}, 14:173--177, 1975.

\bibitem{MR3801483}
M.~van Pruijssen.
\newblock Multiplicity free induced representations and orthogonal polynomials.
\newblock {\em Int. Math. Res. Not. IMRN}, (7):2208--2239, 2018.

\bibitem{MR447979}
L.~Vretare.
\newblock Elementary spherical functions on symmetric spaces.
\newblock {\em Math. Scand.}, 39(2):343--358 (1977), 1976.

\end{thebibliography}
\bibliographystyle{plain}

\end{document}